\documentclass[hidelinks,11pt, reqno]{amsart}
\usepackage{amsmath}
\usepackage{amssymb,amscd}
\usepackage{graphicx}
\usepackage{extarrows}
\usepackage{latexsym} 
\usepackage{tikz-cd}
\usepackage{enumitem}
\usepackage{float}
\usepackage[colorlinks]{hyperref}
\usepackage{verbatim} 
\usepackage{amsfonts}
\usepackage{mathrsfs}
\usepackage{nccmath}
\usepackage{mathtools}
\usepackage{tabularray}

\usepackage[T1]{fontenc}
\usepackage{tgtermes}

\usepackage{etoolbox}
\usepackage{titlecaps}
\patchcmd{\subsection}{\bfseries}{\titlecap}{}{}
\patchcmd{\subsection}{-.5em}{.5em}{}{}

\usepackage{xpatch}
\makeatletter   
\xpatchcmd{\@tocline}
{\hfil\hbox to\@pnumwidth{\@tocpagenum{#7}}\par}
{\ifnum#1<0\hfill\else\dotfill\fi\hbox to\@pnumwidth{\@tocpagenum{#7}}\par}
{}{}
\makeatother 

\makeatletter
 \def\l@subsection{\@tocline{2}{0pt}{4pc}{6pc}{}}
\def\l@subsubsection{\@tocline{3}{0pt}{8pc}{8pc}{}}
 \makeatother

\numberwithin{equation}{section}

\newcommand{\C}{\mathbb{C}}
\newcommand{\N}{\mathbb{N}}
\newcommand{\R}{\mathbb{R}}
\newcommand{\Z}{\mathbb{Z}}
\newcommand{\T}{\mathbb{T}}

\DeclareMathOperator{\pr}{Pr}

\DeclareMathOperator{\ann}{Ann}
\DeclareMathOperator{\img}{img}
\DeclareMathOperator{\en}{End}

\DeclareMathOperator{\spn}{Span}

\DeclareMathOperator{\Hom}{Hom}
\DeclareMathOperator{\rank}{Rank}
\DeclareMathOperator{\type}{type}
\DeclareMathOperator{\codim}{Codim}

\theoremstyle{plain}
\newtheorem{theorem}{Theorem}[section]
\newtheorem{cor}[theorem]{Corollary}
\newtheorem{lemma}[theorem]{Lemma}
\newtheorem{prop}[theorem]{Proposition}

\theoremstyle{definition}
\newtheorem{definition}[theorem]{Definition}
\newtheorem{remark}[theorem]{Remark}
\newtheorem{example}[theorem]{Example}
\newtheorem{thm}{Theorem}

\newlength\mylen
\newlist{mycases}{enumerate}{1}
\setlist[mycases,1]{label=\textbf{Case~\arabic*.}, 
  labelwidth=\dimexpr-\mylen-\labelsep\relax,leftmargin=0pt,align=right}

\def \mc{\mathcal}

\begin{document}

\title[A Boothby-Wang construction in generalized contact geometry]{A Boothby-Wang construction in generalized contact geometry}
\author[D. Pal]{Debjit Pal}

\address{Institute of Differential Geometry, Gottfried Wilhelm Leibniz Universit\"{a}t Hannover, Germany}

\email{\href{mailto:mathdebjit@gmail.com}{mathdebjit@gmail.com}}



\begin{abstract}
We establish a generalized analogue of the Boothby-Wang theorem in generalized contact geometry, along with related results. We present a general method for constructing examples of generalized contact structures that are not of Poon-Wade type, and even examples that fail to be generalized contact structures. Using Courant reduction methods, we construct a generalized complex structure on a smooth leaf space and equip the generalized contact manifold with a principal bundle structure whose connection is defined by the generalized contact data. Under mild assumptions, we show that the curvature induces a symplectic foliation on the leaf space. Several examples are provided.
\end{abstract}

\maketitle

\renewcommand\contentsname{\vspace{-1cm}}
{
  \hypersetup{linkcolor=black}
  \tableofcontents
}

\hypersetup{
    linkcolor=blue,
    citecolor=magenta,      
    urlcolor=brown
    }

\section{Introduction}\label{intro}
Contact geometry has long been closely related to even-dimensional geometries such as symplectic and complex geometry, especially via normal almost contact structures; see \cite{blair,gei08}. This connection is well described by the Boothby-Wang theorem for regular contact structures \cite{bw}, as well as by results on normal almost contact structures \cite{mori}, as follows.

\begin{thm}\label{thm-1}
~
   \begin{enumerate}
   \setlength\itemsep{0.4em}
       \item(\cite[Theorem 2]{bw}) Let $M$ be a compact regular contact manifold with contact form $\eta$. Then $M$ is a principal $S^1$-bundle over a symplectic manifold, where $\eta$ is a connection form and the symplectic form is given by its curvature.
       \item(\cite[Theorem 1]{mori}) Let $M\rightarrow M_0$ be a principal $S^1$-bundle with a normal almost contact structure $(\varphi,\xi,\eta)$ on $M$ such that $\eta$ is the connection form and $\xi$ the vertical fundamental vector field. Then $M_0$ is a complex manifold and the curvature of the connection is of type $(1,1)$.
   \end{enumerate} 
\end{thm}
In even dimensions, a natural unifying framework is generalized complex (GC) geometry, which includes complex and symplectic structures as its two extreme cases. This notion was introduced by Hitchin \cite{hit1} and later, developed by Gualtieri \cite{Gua, Gua2}. Correspondingly, generalized contact geometry was introduced to study the odd-dimensional analogue, providing a natural generalization of both contact and almost contact structures. This theory was initiated by Iglesias-Ponte et al. \cite{pw2}, and subsequently developed by Vaisman \cite{vais1,vais2}, Wade et al. \cite{pw,wade12}, and Sekiya \cite{sekiya}, particularly for cooriented cases. For more details on generalized contact structures, see \cite{aldi, pw3,vita,wright} and the references therein, especially for the non-cooriented case. However, many aspects of this theory are still open, including an analogue of the Boothby-Wang theorem (Theorem \ref{thm-1}) in this framework.
\vspace{0.5em}

As noted earlier, the Boothby-Wang theorem primarily addresses cooriented contact and normal almost contact structures. For this reason, throughout the article, we adopt the notion of \textit{generalized contact structures} (cf. Definition \ref{cntct def}) introduced by Sekiya \cite{sekiya}, which generalizes the notions of Poon and Wade \cite{pw}, and Vaisman \cite{vais2}, particularly in the coorientable setting. To distinguish the latter, namely the notion of generalized contact structures defined by Poon and Wade \cite{pw}, we refer to them as \textit{generalized contact structures of Poon-Wade type}; see Definition \ref{def:p-w}. Coorientable contact structures (Example \ref{contact eg}), cosymplectic manifolds (Example \ref{cosym eg}), and normal almost contact structures (Example \ref{almost eg}) are natural examples, giving generalized contact structures of Poon-Wade type. However, examples not of Poon-Wade type were previously absent in the literature. This vector field $R$ plays a similar role to that of the Reeb vector field of a classical contact manifold.
\vspace{0.5em}

In this article, we establish a complete analogue of Theorem \ref{thm-1} in generalized contact geometry (see Theorem \ref{thmc2}) and provide a general construction for producing examples of generalized contact structures that are not of Poon-Wade type (cf. Examples~\ref{neweg2}-\ref{neweg2.2}). Moreover, this construction yields examples that are not generalized contact structures; see Example \ref{new eq1}. A generalized almost contact structure on an odd-dimensional manifold $M$ determines a nowhere vanishing vector field $R$, a $1$-form $\eta$ with $\eta(R)=1$. It induces a one-dimensional integrable foliation $\mathcal{R}$ as well as two maximal isotropic subbundles $L_+,L_-\subset(TM\oplus T^*M)\otimes\C$ (cf. \eqref{contact dirac}). The integrability of these subbundles characterizes the associated generalized contact structure; see Definition \ref{def:contct}. 
\vspace{0.5em}

The main contribution of this article is in adapting methods from contact geometry and the reduction theory of Courant algebroids (cf. \cite{burst}) to construct a generalized complex (GC) structure on the leaf space $M/\mathcal{R}$ of $\mathcal{R}$ and to endow $M$ with a principal bundle structure whose connection is given by $\eta$. Under a mild additional condition, we also show that the curvature induces a symplectic foliation on $M/\mathcal{R}$. To do so, an assumption that the leaf space $M/\mathcal{R}$ of $\mathcal{R}$ is a smooth manifold and that $R$ is complete, is often necessary. Precisely, we establish the following result.

\begin{thm}(Theorem \ref{thmc2})
    Let $M$ be a regular complete generalized contact manifold with $[R,\eta]=0$ and $L_+$ as a Courant involutive bundle. Suppose that $H_+:=L_+\cap(TM\oplus\ann(R))\otimes\C$ is a Courant involutive subbundle. Let $K_+:=\pr_{TM}\left((H_++(T\mathcal{R}\otimes\C))/T\mathcal{R}\otimes\C\right)$ where $\pr_{TM}: (TM \oplus T^{\ast}M)\otimes\C\longrightarrow TM\otimes\C$ denotes the natural projection. Then
    \vspace{0.2em}
    \begin{enumerate}
    \setlength\itemsep{0.5em}
    \item $M/\mathcal{R}$ admits a generalized complex (GC) structure $\mathcal{J}_{red}$.
        \item The projection $M\xrightarrow{\Psi} M/\mathcal{R}$ defines a principal $\R$-or $S^1$-bundle structure on $M$ over the generalized complex (GC) manifold $(M/\mathcal{R},\mathcal{J}_{red})$, with $\eta$ as a connection $1$-form and curvature $\omega$ given by $d\eta=\Psi^*\omega$.
        \item If $\gamma\in(0,\infty]$ is the minimal period of the globally defined flow $e^{tR}$, then the cohomology class $[\omega/\gamma]\in H^2(M/\mathcal{R},\Z)\subset H^2(M/\mathcal{R},\R)$ represents the Euler class of the bundle.
         \item If $d\eta|_{\Delta_+}$ is non degenerate at each point of $M$, the curvature $\omega$ determines a symplectic foliation (possibly singular) whose complementary foliation is given by $\ker(\omega)$\,, where $$\Delta_+\otimes\C=K_+\cap\overline{K_+}\,.$$
    \end{enumerate}
Furthermore, when $\dim_\R\Delta_+=0$, $M/\mathcal{R}$ is a complex manifold, and $\omega$ is of type $(1,1)$. Also, $\mathcal{J}_{\mathrm{red}}$ is induced by the complex structure, upto a $B$-field transformation. 
\end{thm}

The condition $[R,\eta]=0$ ensures that $\eta$ is basic with respect to the foliation $\mathcal{R}$. Moreover, in Examples \ref{contact1}-\ref{nrml1}, we show that Theorem \ref{thm-1} is a special case of Theorem \ref{thmc2} when the generalized contact structure is induced by a contact form or by a normal almost contact structure. In the case of a cosymplectic manifold, Example \ref{cosymp1} shows that Theorem \ref{thmc2} yields a flat bundle where $M/\mathcal{R}$ is symplectic, and the GC structure is induced by the symplectic form (cf. Example \ref{symplectic eg}).  
\vspace{0.5em}

Now, the proof of Theorem \ref{thmc2} can be divided into two parts. First, we establish a principal bundle structure on $M$. Following the approach in \cite{grabo25}, we show in Theorem \ref{thm1} that such a structure exists whenever the map $M\xrightarrow{\Psi} M/\mathcal{R}$ is a fibration, even if $M$ is only a regular generalized almost contact manifold satisfying $[R,\eta]=0$. Next, assuming $R$ is complete, we prove that the leaves of $\mathcal{R}$ are diffeomorphic either to $\R$ or to $S^1$. In this case, the fibration assumption on $\Psi$ is no longer necessary, since applying Theorem \ref{thm1} locally shows that the global flow generated by $R$ induces the desired principal bundle structure where $[\omega/\gamma]$ represents the Euler class and $\gamma\in(0,\infty]$ is the minimal period of the global flow; see Theorem \ref{thm2}.
\vspace{0.5em}

In the second part, we construct the GC structure $\mathcal{J}_{red}$ on $M/\mathcal{R}$. The main idea of this construction is to obtain a maximal isotropic subbundle $L_{red}\subset(T(M/\mathcal{R})\oplus T^*(M/\mathcal{R}))\otimes\C$ that is involutive with respect to the Courant bracket on $M/\mathcal{R}$, and that induces an orthogonal (with respect to the bilinear form \eqref{bilinear} on $M/\mathcal{R}$) splitting of $(T(M/\mathcal{R})\oplus T^*(M/\mathcal{R}))\otimes\C$; cf. Definition \ref{gcs}. To do this, first it is essential to consider a subbundle $E\subset TM\oplus T^*M$ for which the restricted Courant bracket is closed on basic sections of $E$ (cf. Definition \ref{imp def}) with respect to the foliation. This ensures a direct connection with the standard Courant algebroid $T(M/\mathcal{R})\oplus T^*(M/\mathcal{R})$. We then need to find an appropriate Dirac structure (that is, a maximal, isotropic Courant involutive subbundle) of $E\otimes\C$, and show that it descends to $M/\mathcal{R}$. For further details on Courant algebroids, we refer to \cite{burst} and the references therein.
\vspace{0.5em}

By the reduction theory of Courant algebroids, a natural choice for $E$ is $\mathcal{N}_R\oplus\ann(R)$ where $\mathcal{N}_R$ is the normal bundle of $\mathcal{R}$, and the corresponding Dirac structure is $(H_++(T\mathcal{R}\otimes\C))/T\mathcal{R}\otimes\C$, viewed as a subbundle of $(TM\oplus\ann(R))\otimes\C/T\mathcal{R}\otimes\C\cong(\mathcal{N}_R\oplus\ann(R))\otimes\C$, since both bundles are Courant involutive on basic sections (cf. Proposition \ref{prop  c1} and Equation \eqref{1.2}). Next, we need to identify the fibers of these bundles along the leaves of $\mathcal{R}$. For this, it is necessary that $H_+$ be a subbundle and that $H_++(T\mathcal{R}\otimes\C)$ be Courant involutive. Proposition \ref{prop c} shows that $H_+$ is a always a subbundle, and we prove in Theorem \ref{thm c} that $H_++(T\mathcal{R}\otimes\C)$ is Courant involutive if and only if $H_+$ is so. By extending the Bott connection (cf. \cite{bott71}) from $\mathcal{N}_R$ to $\mathcal{N}_R\oplus\ann(R)$ and $(H_++(T\mathcal{R}\otimes\C))/T\mathcal{R}\otimes\C$, we can apply Lemma \ref{imp lemma1} for the appropriate identification of the fibers, which shows that the bundles $\mathcal{N}_R\oplus\ann(R)$ and $(H_++(T\mathcal{R}\otimes\C))/T\mathcal{R}\otimes\C$ descend to $T(M/\mathcal{R})\oplus T^*(M/\mathcal{R})$ and the desired Dirac structure $L_{red}$ over $M/\mathcal{R}$, respectively. The Courant involutiveness of $H_++(T\mathcal{R}\otimes\C)$ is essential for restricting the partial connection (Equation \ref{eq1.4}). Therefore, we prove the following result, which completes the proof of Theorem \ref{thmc2}.

\begin{thm}(Theorem \ref{thm c1})
Let $M$ be a generalized contact manifold with Courant involutive bundle $L_+$. Assume that $H_+:=L_+\cap(TM\oplus\ann(R))\otimes\C$ is Courant involutive. Then,
\vspace{0.2em}
\begin{enumerate}
\setlength\itemsep{0.3em}
    \item Both vector bundles  $\mathcal{N}_R\oplus\ann(R)$ and $(H_++(T\mathcal{R}\otimes\C))/T\mathcal{R}\otimes\C$ admit partial connections along $\mathcal{R}$, that is, $T\mathcal{R}$-connections, which are flat and have trivial holonomy.
    \item $(H_++(T\mathcal{R}\otimes\C))/T\mathcal{R}\otimes\C$ induces a generalized complex structure on $(\mathcal{N}_R\oplus\ann(R))\otimes\C$.
    \item If $\mathcal{R}$ is simple, $(H_++(T\mathcal{R}\otimes\C))/T\mathcal{R}\otimes\C$ descends to a GC structure on $M/\mathcal{R}$.
    \end{enumerate}
\end{thm}
We further prove a converse to Theorem \ref{thmc2} in Subsection \ref{convers}, which recovers as special cases the classical converse results associated with Theorem \ref{thm-1}; see Remark \ref{spcl rmk}. Specifically, we obtain the following result (Theorem \ref{conn thm}): Any principal $S^1$-bundle over a GC manifold manifold, admits a canonical regular generalized contact structure of Poon-Wade type. 
\medskip

Finally, in Subsection \ref{non pw}, we present a general method for constructing examples of generalized contact structures that are not of Poon-Wade type. In particular, we obtain new examples on $\mathbb{R}^{2n+1}$ (Example \ref{neweg2}) and on the three-dimensional Heisenberg group (Example \ref{new eq1}). Moreover, in Example \ref{neweg2.2}, we show that $\mathbb{T}^2 \times S^1$ admits generalized contact structures that are neither product type (cf. Example \ref{prodct1}) nor Poon-Wade type.

\section{Preliminaries}\label{prelim}
\subsection{Generalized complex structures} We first start by recalling some basic notions of generalized complex (in short GC) geometry. In this subsection, we shall rely upon  \cite{Gua} and \cite{Gua2} for most of the definitions and results.  

\vspace{0.3em}
Given any smooth manifold $M$, the direct sum of tangent and cotangent bundles of $M$, which we denote by $TM\oplus T^{*}M$, is endowed with a natural symmetric bilinear form,
\begin{equation}\label{bilinear}
    \langle X+\xi,Y+\beta\rangle\,:=\,\frac{1}{2}(\xi(Y)+\beta(X))\,.
\end{equation}
It is also equipped with the \textit{Courant Bracket}  defined as follows.
\begin{definition}
Given a real closed form $H\in\Omega^3(M)$, the $H$-twisted Courant bracket is a skew-symmetric bracket defined on smooth sections of $TM\oplus T^{*}M$, given by
\begin{equation}\label{bracket}
    [X+\xi,Y+\beta]_H := \mathcal{L}_{X}Y+\mathcal{L}_{X}\beta-\mathcal{L}_{Y}\xi-\frac{1}{2}d(i_{X}\beta-i_{Y}\xi)+i_Yi_X H,
\end{equation}  
where $X,Y\in C^{\infty}(TM)$, $\xi,\beta\in C^{\infty}(T^{*}M)$, and $\mathcal{L}_{X},\,\, i_{X}$ denote the Lie derivative and the interior product of forms with respect to the vector field $X$, respectively. For $H=0$, we simply refer to it as the Courant bracket and denote it by $[\cdot\,,\cdot\,]$. 
\end{definition}
Let $\pr_{TM}: (TM \oplus T^{\ast}M) \longrightarrow TM $ denote the natural projection. The complexified version of $\pr_{TM}$ will also be denoted by $\pr_{TM}$ for simplicity of notation. There are natural relations obtained by combining the Courant bracket with the symmetric bilinear pairing, as follows.
\vspace{0.2em}
\begin{equation}\label{eqcc}
   \begin{aligned}
   \bullet\,\,&\pr_{TM}(s_1)\langle s_2,s_3\rangle=\langle [s_1,s_2],s_3\rangle +\langle s_2,[s_1,s_3]\rangle;\\
       \bullet\,\,&[[s_1, s_2], s_3] + [s_2, [s_1, s_3]]-[s_1, [s_2, s_3]]=\frac{1}{3}d\bigg(\langle [s_1,s_2],s_3\rangle+\langle s_1,[s_2,s_3]\rangle\\
       &\hspace{5.5cm}\quad\quad\quad-\langle s_2,[s_1,s_3]\rangle\bigg)\,.
   \end{aligned}
\end{equation}
where $s_1,s_2,s_3\in C^\infty(TM \oplus T^{\ast}M)$ and $d$ is the exterior derivative.
\begin{remark}\label{rmk notatn}
  Note that, an endomorphism $\mathcal{J}_M$ of $TM\oplus T^*M$ can be written in a matrix form as
$$\mathcal{J}_M(X+\xi)=\begin{pmatrix}
 J & \pi\\
b & \hat{J}    
\end{pmatrix}
\begin{pmatrix}
 X\\
 \xi   
\end{pmatrix}=\begin{pmatrix}
 J(X)+\pi(\xi)\\
b(X)+\hat{J}(\xi)   
\end{pmatrix}$$
where $J\in\en(TM)$, $\hat{J}\in\en(T^*M)$, $b\in\Hom(TM,T^*M)$ and $\pi\in\Hom(T^*M,TM)$. Here, the notations $\en(\bullet)$ and $\Hom(\bullet,\bullet)$ represent the standard endomorphism and homomorphism bundles over $M$, respectively. Identifying $TM\oplus T^*M$ with its dual bundle $(TM\oplus T^*M)^*$ via the bilinear form in \eqref{bilinear}, the dual map $\mathcal{J}_{M}^{*}\in\en(TM\oplus T^*M)$ of $\mathcal{J}_{M}$ is defined by $$\mathcal{J}_{M}^*(Y+\beta)(X+\xi)=\langle \mathcal{J}_{M}(X+\xi),Y+\beta\rangle\,.$$ Furthermore, if condition $\mathcal{J}_{M}^{*}=-\mathcal{J}_{M}$ holds, then $b$ is a $2$-form, $\pi$ is a bi-vector field, and $\hat{J}=-J^*$ where the operator $J^*:C^\infty(T^*M)\rightarrow C^\infty(T^*M)$, the dual map of $J$, is defined as
$$(J^*\xi)(X)=\xi(J(X))\,.$$
In particular, for any $\phi\in\en(TM)$, its dual map $\phi^*$ is defined \textit{verbatim} as for $J$. We will make extensive use of the preceding matrix expression and the notion of the dual map in the remainder of this section and throughout later sections. For ease of reference, the matrix expression may be repeated when necessary. For example, see Formula \eqref{Phi}. 
\end{remark}

We are now ready to define the notion of generalized complex (GC) structures on a $2n$-dimensional smooth manifold $M$ in two equivalent ways.
\begin{definition}(cf. \cite{Gua})\label{gcs}
A \textit{generalized complex structure} (in short GCS) is determined by any of the following two equivalent sets of data:

\vspace{0.5em}
\begin{enumerate}
\setlength\itemsep{1em}
    \item  A bundle automorphism $\mathcal{J}_{M}$ of $TM\oplus T^{*}M$ which satisfies the following conditions:
    \vspace{0.2em}
    \begin{itemize}
 \setlength\itemsep{1em}
        \item[(a)] $\mathcal{J}_{M}^{2}=-I_{(TM\oplus T^*M)}$ where $I_{(TM\oplus T^*M)}$ is the identity automorphism of $(TM\oplus T^*M)$.
        \item[(b)] $\mathcal{J}_{M}^{*}=-\mathcal{J}_{M}$, that is, $\mathcal{J}_{M}$ is orthogonal with respect to the natural pairing in \eqref{bilinear}
        \item[(c)] $\mathcal{J}_{M}$ has vanishing {\it Nijenhuis tensor}, that is, for all $C, D \in C^{\infty}(TM\oplus T^{*}M)$,
   $$ N(C, D) :=[\mathcal{J}_{M}C, \mathcal{J}_{M}D]-\mathcal{J}_{M} [\mathcal{J}_{M}C, D] - \mathcal{J}_{M} [C, \mathcal{J}_{M} D] 
    - [C, D]=0\,. $$ 
    \end{itemize}
     \item A subbundle, say $L_{M}$, of $(TM\oplus T^{*}M)\otimes\C$ which is maximal isotropic with respect to the natural bilinear form \eqref{bilinear}, involutive with respect to the Courant bracket \eqref{bracket}, and satisfies $L_{M}\cap\overline{L_{M}}=\{0\}$.
\end{enumerate}
\end{definition}
In Definition \ref{gcs}, the two equivalent conditions are related to each other by the fact that the subbundle $L_M$ may be obtained as the $+i$-eigenbundle of the automorphism $\mathcal{J}_M$.

\vspace{0.5em}
Given any GC manifold  $(M,\,\mathcal{J}_{M})$, we can deform $\mathcal{J}_{M}$ by a real closed $2$-form $B$, known as a \textit{$B$-field transformation}, to get another GCS on $M$,  \begin{equation}\label{B transformation}
        (\mathcal{J}_{M})_{B}=e^{-B}\circ\mathcal{J}_{M}\circ e^{B}\quad\text{where}\quad e^{B}=\begin{pmatrix} 
	       1 & 0 \\
	       B & 1 \\
	    \end{pmatrix}\,.
    \end{equation}
The $+i$-eigenbundle of $(\mathcal{J}_{M})_{B}$ is 
\begin{equation}\label{L-B}
(L_{M})_{B}=\{X+\xi-B(X,\,\cdot)\,|\,X+\xi\in L_{M}\}\,.    
\end{equation}
Let us consider some simple examples of GCS.
\begin{example}\label{complx eg}
Let $(M,\, J_{M})$ is a complex manifold with a complex structure $J_{M}$. Then the natural GCS on $M$ is given by the bundle automorphism  
\[
\mathcal{J}_{M}:=
\begin{pmatrix}
 J_{M}     &0 \\
    
    0        &-J^{*}_{M}
\end{pmatrix}: TM\oplus T^{*}M\longrightarrow TM\oplus T^{*}M\,.
\] Its corresponding $+i$-eigen bundle is
$$L_{M}=T^{1,0}M\oplus(T^{0,1}M)^{*}\,.$$
\end{example}
\begin{example}\label{symplectic eg}
Let $(M,\,\omega)$ be a symplectic manifold with a symplectic structure $\omega$. Then, the bundle automorphism   
\[
\mathcal{J}_{M}:=
\begin{pmatrix}
    0    &-\omega^{-1}\\
    \omega    &0
\end{pmatrix}: TM\oplus T^{*}M\longrightarrow TM\oplus T^{*}M\,,
\] gives a natural GCS on $M$. The $+i$-eigen bundle of this GCS is 
$$L_{M}=\{X-i\omega(X)\,|\,X\in TM\otimes\C\}\,.$$
\end{example}  

\begin{definition}\label{def:type} Let $\Delta_M \otimes \C := E_M \bigcap \overline{E_M}$ where $E_M:=\pr_{TM}(L_M)$. Here $\pr_{TM}$ is the projection on $(TM\oplus T^{*}M)\otimes\C$. For each $x\in M$, type of $\mathcal{J}_{M}$ at $x$ is defined as 
    $$\type(x):=\codim_{\C}((E_{M})_{x})=\frac{1}{2}\codim_{\R}((\Delta_{M})_{x})\,.$$ A point $x\in M$ is called a regular point of $M$ if $\type(x)$ is constant in a neighborhood of $x$ and $M$ is called a regular GC manifold if each point of $M$ is a regular point. Note that, $0\leq\type(x)\leq\frac{\dim_{\R}M}{2}$.
\end{definition}
 We have the generalized Darboux theorem around any regular point.
\begin{theorem}(\cite[Theorem 4.3]{Gua2})\label{darbu thm}
For a regular point $x\in (M,\mc{J}_{M})$ of $\type(x)=k$, there exists an open neighborhood $U_{x}\subset M$ of $x$ such that, after a $B$-transformation, $(U_{x},\mc{J}_{M}|_{U_{x}})$ is diffeomorphic to $U_1\times U_2$ as GC manifolds, where $U_1\subset(\R^{2n-2k},\omega_0), U_2\subset\C^{k}$ are  open subsets with $\omega_0$ being the standard symplectic structure, and $U_1\times U_2$ is obtained with the product GCS.
\end{theorem}

\subsection{Generalized contact structures}
This subsection revisits some basic concepts in generalized contact geometry. We follow \cite{sekiya} for most of the definitions, results and notational conventions.
 \begin{definition}(cf. \cite{sekiya})\label{cntct def}
     A \textit{generalized almost contact structure} on a smooth (odd-dimensional) manifold $M$ is a triple $(\Phi,E_{\pm})$ where $\Phi\in\en(TM\oplus T^*M)$ and $E_{\pm}\in C^\infty(TM\oplus T^*M)$, which satisfy
     \begin{itemize}
     \setlength\itemsep{0.5em}
         \item $\Phi+\Phi^*=0$, $\langle E_+,E_-\rangle=\frac{1}{2}$, and $\langle E_{\pm},E_{\pm}\rangle=0$;
         \item $\Phi^2=-I_{(TM\oplus T^*M)}+E_+\otimes E_-+E_-\otimes E_+$.
     \end{itemize}
 where
$$E_+\otimes E_-+E_-\otimes E_+=
\begin{pmatrix}
    \eta_+\otimes X_-+\eta_-\otimes X_+ & X_+\otimes X_-+X_-\otimes X_+\\
    \eta_+\otimes\eta_-+\eta_-\otimes\eta_+ & X_+\otimes\eta_-+X_-\otimes\eta_+
\end{pmatrix}\,\,\,\text{if}\,\,\,E_{\pm}=X_{\pm}+\eta_{\pm}\,.$$ In other words, this means that, for all $X+\xi\in C^{\infty}(TM\oplus T^*M)$,
$$\Phi(\Phi(X+\xi))=-(X+\xi)+2\left(\langle E_-,X+\xi\rangle E_++\langle E_+,X+\xi\rangle E_-\right)\,.$$
We refer to the triple $(M, \Phi, E_{\pm})$ as a generalized almost contact manifold, or simply to $M$ when the generalized almost contact structure is understood from the context.
 \end{definition}
\begin{remark}\label{imp rmk2}
~
\begin{enumerate}
\setlength\itemsep{0.5em}
    \item Note that $\{X_+,X_-\}$ is a linearly independent set in Definition \ref{cntct def} whenever $X_+,X_-$ both are nowhere vanishing. To see this, if possible let, there exist nowhere vanishing smooth function $f:M\longrightarrow\R$ such that $X_+=f\,X_-$ on $M$. By assumption $\langle E_{\pm},E_{\pm}\rangle=0$, we have
    $$f\eta_+(X_-)=0\quad\text{and}\quad \frac{1}{f}\eta_-(X_+)=0\,,$$ implying $\eta_+(X_-)=\eta_-(X_+)=0$ as $f\not\equiv 0$. It contradicts the assumption that $\eta_+(X_-)+\eta_-(X_+)=1$. Similarly $\{\eta_+,\eta_-\}$ is also a linearly independent set whenever $\eta_+,\eta_-\not\equiv 0$.
\item By a simple computation, we can see that $\Phi^3=-\Phi$, $\Phi(E_{\pm})=0$ (cf. \cite[Lemma 3.1]{sekiya}) and that, $\langle E_\pm,\Phi(X+\xi)\rangle=0$ for all $X+\xi\in C^\infty(TM\oplus T^{*}M)$.
\end{enumerate}
\end{remark}
Let $L_{E_{\pm}}$ be real line bundle, generated by $E_{\pm}$, respectively. By Definition \ref{cntct def}, we get $\Phi^3+\Phi=0$ implying that $\Phi$ has three eigenvalues, namely $0,\pm i$. Then the corresponding eigenbundles of $0,\pm i$  are as follows, respectively.
\begin{align*}
    &\ker(\Phi)=L_{E_+}\oplus L_{E_-}\subset(TM\oplus T^{*}M)\,;\\
    &E^{(1,0)}:=\{e-i\Phi(e)\,|\,\langle e,E_{\pm}\rangle=0\,\&\,e\in(TM\oplus T^{*}M)\}\subset(TM\oplus T^{*}M)\otimes\C\,;\\
    &E^{(0,1)}:=\{e+i\Phi(e)\,|\,\langle e,E_{\pm}\rangle=0\,\&\,e\in(TM\oplus T^{*}M)\}\subset(TM\oplus T^{*}M)\otimes\C\,.
\end{align*}
Set 
\begin{equation}\label{contact dirac}
    L_{\pm}:=(L_{E_{\pm}}\otimes\C)\oplus E^{(1,0)}
\end{equation}
\begin{prop}(\cite[Lemma 3.3]{sekiya}) 
    The complex vector bundles $E^{(1,0)},E^{(0,1)},L_{\pm},\overline{L_{\pm}}$ are all isotropic subbundles of $(TM\oplus T^{*}M)\otimes\C$ with respect to the bilinear form in \eqref{bilinear}. In particular, the subbundles $L_{\pm},\overline{L_{\pm}}$ are maximal isotropic.
\end{prop}
\begin{definition}(\cite[Definition 3.2]{sekiya})\label{def:contct}
 A generalized almost contact structure $(\Phi,E_{\pm})$  on $M$ is called a \textit{generalized contact structure} if either of the maximal isotropic subbundles $L_{\pm}$ in \eqref{contact dirac} is involutive with respect to the Courant bracket \eqref{bracket}. It is called a \textit{strong generalized contact structure} if both $L_+$ and $L_-$ are Courant involutive. If, in addition, a strong generalized contact structure satisfies $[E_+,E_-]=0$, then it is called a normal generalized contact structure. 
 A \textit{generalized contact manifold} is a generalized almost contact manifold $(M, \Phi, E_{\pm})$ such that $(\Phi, E_{\pm})$ is a generalized contact structure on $M$.
\end{definition}
The following proposition, due to Gomez et al. \cite{gomez}, provides a necessary and sufficient condition for a generalized almost contact structure to be strong.
\begin{prop}\label{strng}
    A generalized almost contact structure $(\Phi, E_{\pm})$ on $M$ is strong if and only if  $(\Phi, E_{\pm})$ satisfies $[C^\infty(L_\pm), C^\infty(E^{(1,0)})]\subseteq C^\infty(E^{(1,0)})$.
\end{prop}
\begin{example}\label{B-contact}
    Let $(M, \Phi, E_{\pm})$ be a generalized almost contact manifold then $B$-field transformation of $(\Phi, E_{\pm})$, namely $(e^B\Phi e^{-B},e^B E_{\pm})$ yields another generalized almost contact structure on $M$; see \cite[Lemma 3.2]{sekiya}. Here $e^{B}$ as defined in \eqref{B transformation}. As $B\in\Omega^2(M)$ is a real closed form, it preserve the Courant bracket. Thus whenever $(M, \Phi, E_{\pm})$ is a generalized contact manifold, so is $(M,e^B\Phi e^{-B},e^B E_{\pm})$. Moreover, even if $B$ is not closed, we obtain only that $(e^B\Phi e^{-B},e^B E_{\pm})$ defines a generalized almost contact structure. In this case, however, $(e^B\Phi e^{-B},e^B E_{\pm})$ is a generalized contact structure with respect to the $(-dB)$-twisted Courant bracket; see Definiton \ref{bracket}.
\end{example}
\begin{example}(\cite{pw,sekiya})\label{contact eg}
Let $(M,\xi,\alpha)$ be a cooriented contact manifold where $\alpha$ is the contact form and $\xi$ denotes its associated Reeb vector field. For $\beta,\beta'\in T^*M$, define a bivector $\pi$ by
$$\pi(\beta,\beta'):=d\alpha(\delta^{-1}(\beta),\delta^{-1}(\beta'))\,,$$ where $\delta:TM\rightarrow T^*M$ is an isomorphism, defined by $$\delta(X)=i_Xd\alpha-\alpha(X)\alpha\,.$$ Then, we have a generalized almost contact structure by setting
$$E_+=\alpha,\,\,E_-=\xi,\,\,\Phi=\begin{pmatrix}
    0    &\pi\\
    d\alpha    &0
\end{pmatrix}: TM\oplus T^{*}M\longrightarrow TM\oplus T^{*}M\,.$$
In fact, $(M,\xi,\alpha)$ is a generalized contact manifold which is not strong; see \cite[Proposition~3.1]{pw}.
\end{example}
\begin{example}(\cite{pw})\label{almost eg}
Let $(M^{2n+1},\varphi,\xi,\alpha)$ be an almost contact manifold, satisfying
$$\varphi^2=-I_{TM}+\alpha\otimes\xi\quad\text{and}\quad\alpha(\xi)=1\,,$$ where $\varphi$ is a type $(1,1)$-tensor. Then the generalized almost contact structure is given by $$E_+=\alpha,\,\,E_-=\xi,\,\,\Phi=\begin{pmatrix}
    \varphi    &0\\
     0   &-\varphi^*
\end{pmatrix}: TM\oplus T^{*}M\longrightarrow TM\oplus T^{*}M\,.$$
By \cite[Theorem 3.4]{wade12}, $(\Phi,E_\pm)$ is a generalized contact structure for which $L_-$ is Courant involutive if and only if  $N_\varphi=-\xi\otimes d\alpha$ and $\mathcal{L}_\xi\varphi=0$, where, for $X,Y\in C^{\infty}(TM)$, $N_\varphi$ is defined by
$$N_\varphi(X,Y)=[\varphi(X),\varphi(Y)]+\varphi^2([X,y])-\varphi([\varphi(X),Y]+[X,\varphi(Y)])\,.$$
In particular, if $(M^{2n+1},\varphi,\xi,\alpha)$ is a normal almost contact manifold (cf. \cite[Chapter~6]{blair}), then $(M,\Phi,E_\pm)$ is a normal generalized contact manifold; see \cite[Proposition 3.4]{pw}.
\end{example}
\begin{example}\label{cosym eg}
Let $(M^{2n+1},\alpha,\theta)$ be an almost cosymplectic manifold where $\alpha\in\Omega^1(M),\theta\in\Omega^2(M)$ hold the condition $\alpha\wedge\theta^n\neq 0$ on $M$. Then $M$ admits a generalized almost contact structure $(\Phi,E_\pm)$ defined by
$$E_+=\alpha,\,\,E_-=\xi,\,\,\Phi=\begin{pmatrix}
    0    &\pi\\
    \theta    &0
\end{pmatrix}: TM\oplus T^{*}M\longrightarrow TM\oplus T^{*}M\,,$$ where $\xi$ is the unique vector field associated with $\alpha$ such that $\alpha(\xi)=1$ and $\theta(\xi)=0$, and $\pi$ is defined as in Example \ref{contact eg}, with $d\alpha$ replaced by $\theta$. By \cite[Theorem 3.4]{wade12}, $(M,\Phi,E_\pm)$ is a generalized contact manifold if and only if $\theta$ is closed.
\vspace{0.3em}

When $d\alpha=d\theta=0$, it is called a cosymplectic manifold; see \cite[Section~6.5]{blair}. So, a cosymplectic manifold $(M^{2n+1},\alpha,\theta)$ always admits the strong generalized contact structure $(\Phi,E_\pm)$ by \cite[Proposition 3.2]{pw}. 
\end{example}
Another natural example of generalized contact structures arises from product structures, as given in the following example.
\begin{example}\label{prodct}
Let $(M, \Phi, E_\pm)$ be a generalized contact manifold, and let $(N, \mathcal{J}_N)$ be a generalized complex (GC) manifold. Using the identification
$$T(M\times N)\oplus T^*(M\times N)\cong(TM\oplus T^*M)\oplus(TN\oplus T^*N)\,,$$
one can define a generalized contact structure $(\tilde{\Phi}, \tilde{E}_\pm)$ (cf. \cite{gomez}) on $M\times N$ as follows:
\begin{itemize}
\setlength\itemsep{0.3em}
    \item The endomorphism $\tilde{\Phi}$ is given by $$\tilde{\Phi}:= \Phi \oplus \mathcal{J}_N\,,$$ acting componentwise on $T(M\times N)\oplus T^*(M \times N)$.
    \item The sections $\tilde{E}_\pm$ are obtained by lifting $E_\pm$ via the natural projection $M \times N\to M$ together with a inclusion $M \hookrightarrow M\times N$. 
\end{itemize}
\end{example}
\begin{example}\label{prodct1}
   $\R$ admits a generalized contact structure $(\Phi, E_\pm)$ where $\Phi\equiv0$, and $(E_+,E_-)=(dt,\frac{\partial}{\partial t})$. Then for any GC manifold $(N, \mathcal{J}_N)$, the triplet $(\mathcal{J}_N, dt,\frac{\partial}{\partial t})$ is a generalized contact structure on the product $N\times\R$ by Example \ref{prodct}. 
\end{example}
\begin{definition}(cf. \cite{pw})\label{def:p-w}
A generalized contact manifold $(M,\Phi,E_{\pm})$ is of \textit{Poon-Wade type} when at least one of the following holds:
\vspace{0.2em}
\begin{enumerate}
\setlength\itemsep{0.2em}
    \item $E_+ \in C^\infty(TM)$ and $E_- \in C^\infty(T^*M)$, with $L_+$ Courant involutive,
    \item $E_- \in C^\infty(TM)$ and $E_+ \in C^\infty(T^*M)$, with $L_-$ Courant involutive.
\end{enumerate}
In a similar manner, \textit{generalized almost contact manifolds of Poon-Wade type} is defined without the Courant involutiveness condition.
\end{definition}
\begin{example}
Note that the generalized contact structures in Example \ref{contact eg}, Example \ref{almost eg}, and Example \ref{cosym eg} are of Poon-Wade type. A natural way to construct a generalized contact structure that is not of Poon-Wade type is via a $B$-field transformation. More precisely, start with a generalized contact manifold $(M,\Phi,E_\pm)$ of Poon-Wade type, and consider a real closed $2$-form $B$ that does not vanish on the vector field associated with the Poon-Wade type generalized contact structure $(\Phi,E_\pm)$. The resulting generalized contact structure $(e^B\Phi e^{-B},e^B E_{\pm})$ is then no longer of Poon-Wade type.    
\end{example}
Now, consider the following four cases for any generalized almost contact manifold $(M,\Phi,E_{\pm})$ with $E_\pm=X_\pm+\eta_\pm$.
\vspace{0.2em}
\begin{mycases}
\setlength\itemsep{0.5em} 
    \item Either $X_\pm=0$ or $\eta_\pm=0$. This is not possible as $\eta_+(X_-)+\eta_-(X_+)=1$.
    \item $X_+=0,X_-\ne 0$. It follows that $\eta_+\ne 0$. However $\eta_-$ may either vanish or be nonzero. Let $B=\eta_-\wedge\eta_+$; note that $B$ may vanish. Then $(M,e^B\Phi e^{-B},e^B E_{\pm})$ defines a generalized almost contact structure of Poon-Wade type; see Example \ref{B-contact}. Moreover, if $(M,\Phi,E_{\pm})$ is a generalized contact structure, then $(M,e^B\Phi e^{-B},e^B E_{\pm})$ is also a generalized contact structure with respect to the twisted Courant bracket $[\cdot\,,\cdot]_{-dB}$.
    \item $X_-=0,X_+\ne 0$. This is similar to Case $2$.
    \item $X_+\ne 0,X_-\ne 0$. This implies that $\eta_\pm=0$ cannot occur. Consequently, there are only three possible choices, namely,
    $$(1)\,\,\eta_+\ne 0\,\,\text{and}\,\,\eta_-=0;\quad(2)\,\,\eta_-\ne 0\,\,\text{and}\,\,\eta_+=0;\quad(3)\,\,\eta_+\ne 0\,\,\text{and}\,\,\eta_-\ne 0\,.$$
The first two choices are similar, so it is sufficient to consider only the second and third. Let $\tilde{\pi}:=X_+\wedge X_-$ be a bivector, and consider the orthogonal transformation (with respect to \eqref{bilinear}), defined by
$$e^{\tilde{\pi}}=\begin{pmatrix} 
	       1 & \tilde{\pi}\\
	       0 & 1 \\
	    \end{pmatrix}: TM\oplus T^{*}M\longrightarrow TM\oplus T^{*}M\,.$$
This is also known as $\beta$-field transformation; see \cite[Example 2.2]{Gua}. Then we can see that $(e^{-\tilde{\pi}}\Phi e^{\tilde{\pi}},e^{-\tilde{\pi}} E_{\pm})$ is a generalized almost contact structure of Poon-Wade type, in the case of the second choice. However, in this case, even if $(\Phi,E_{\pm})$ defines a generalized contact structure, its $\beta$-transformation need not do so. Indeed, $\beta$-transformations are not, in general, symmetries of the Courant bracket. Thus we have proved the following.
\end{mycases}
\begin{prop}
   Let $(M,\Phi,E_{\pm})$ be a generalized almost contact manifold, excluding the case where both $X_\pm$ and $\eta_\pm$ are nonzero when $E_\pm=X_\pm+\eta_\pm$. Then $(\Phi,E_{\pm})$ is obtained from a generalized almost contact structure of Poon-Wade type via a $\beta$- or $B$-field transformation. Here $B$ is not necessarily closed form.
\end{prop}
For the third choice, it is possible that the generalized almost contact structure cannot be derived from a generalized almost contact structure of Poon-Wade type. The following example demonstrates this in a simple case.
\begin{example}(New generalized contact structures on $\R^3$ and $\T^3$)\label{new eq}
    Consider $\R^3$ with coordinates $(x,y,z)$. Without loss of generality choose $(x,z)$, and define
    $$E_+:=\frac{\partial}{\partial z}+dx\,\,\,\text{and}\,\,\,E_-:=\frac{1}{2}(dz+\frac{\partial}{\partial x})\,.$$
It is clear that $\langle E_+,E_-\rangle=\frac{1}{2}$, and $\langle E_{\pm},E_{\pm}\rangle=0$. Let $V:=\spn\{E_\pm\}$, and consider its orthogonal complement (with respect to the bilinear form \eqref{bilinear}) $V^\perp$, defined by
$$V^\perp:=\big\{X+\xi\in\C^\infty(T\R^3\oplus T^*\R^3)\,|\,\langle X+\xi,E_\pm\rangle=0\big\}\,.$$ We now proceed to compute $V^\perp$. Let $X+\xi\in V^\perp$, and, for some $f_i,g_i\in C^\infty(\R^3)$ ($i=1,2,3$),
$$X=f_1\frac{\partial}{\partial x}+f_2\frac{\partial}{\partial y}+f_3\frac{\partial}{\partial z}\,,\,\,\text{and}\,\,\xi=g_1dx+g_2dy+g_3dz\,.$$
Then, $\langle X+\xi,E_\pm\rangle=0$ implies that $f_1+g_3=0$ and $f_3+g_1=0$, respectively. Consequently, $$V^\perp=\spn\left\{\frac{\partial}{\partial x}-dz,\frac{\partial}{\partial y},\frac{\partial}{\partial z}-dx,dy\right\}\,.$$
Let $\omega:=dy\wedge dx$, and $\omega^{-1}:=\frac{\partial}{\partial x}\wedge\frac{\partial}{\partial y}$. Consider the endomorphism $\phi$ on $T\R^3\oplus T^*\R^3$, defined as
$$\phi=\begin{pmatrix}
    0    &-\omega^{-1}\\
    \omega    &0
\end{pmatrix}\,,\,\,\,\text{and the projection}\,\,\,\pr:T\R^3\oplus T^*\R^3\rightarrow V^\perp\,.$$
It follows that
\begin{align*}
&(\pr\circ\phi)\left(\frac{\partial}{\partial x}-dz\right)=dy\,,(\pr\circ\phi)\left(\frac{\partial}{\partial y}\right)=\frac{1}{2}\left(\frac{\partial}{\partial z}-dx\right)\,;\\
&(\pr\circ\phi)\left(\frac{\partial}{\partial z}-dx\right)=-\frac{\partial}{\partial y}\,,\,\,\text{and}\,\,(\pr\circ\phi)(dy)=-\frac{1}{2}\left(\frac{\partial}{\partial x}-dz\right)\,.
\end{align*}
Therefore, \begin{align*}
&(\pr\circ\phi)^2\left(\frac{\partial}{\partial x}-dz\right)=-\frac{1}{2}\left(\frac{\partial}{\partial x}-dz\right)\,,(\pr\circ\phi)^2\left(\frac{\partial}{\partial y}\right)=-\frac{1}{2}\frac{\partial}{\partial y}\,;\\
&(\pr\circ\phi)^2\left(\frac{\partial}{\partial z}-dx\right)=-\frac{1}{2}\left(\frac{\partial}{\partial z}-dx\right)\,,\,\,\text{and}\,\,(\pr\circ\phi)^2(dy)=-\frac{1}{2}dy\,.
\end{align*}
Define
\begin{equation*}
\mathcal{J}:=\sqrt{2}(\pr\circ\phi)\,,\quad\text{and}\quad\Phi:=
  \begin{cases}
   \mathcal{J}& \text{on}\quad V^\perp\,;\\
    0 & \text{on}\quad V\,.
\end{cases}  
\end{equation*}
Note that $\Phi^2=-I$ on $V^\perp$. Let $X+\xi\in C^\infty(T\R^3\oplus T^*\R^3)$ such that $X+\xi=a_+E_++a_-E_-+b$ where $b\in V^\perp$. Then
\begin{align*}
&-(X+\xi)+2\left(\langle E_-,X+\xi\rangle E_++\langle E_+,X+\xi\rangle E_-\right)\\
    &=-(a_+E_++a_-E_-+b)+2\left(\langle E_-,a_+E_++a_-E_-+b\rangle E_++\langle E_+,a_+E_++a_-E_-+b\rangle E_-\right)\\
    &=-(a_+E_++a_-E_-+b)+(a_+E_++a_-E_-)\\
    &=-b=\Phi(\Phi(X+\xi))
\end{align*}
Hence $(\Phi,E_\pm)$ defines a generalized almost contact structure on $\R^3$ that is not of Poon-Wade type, even up to $\beta$- or $B$-transformations. In fact, it is straightforward to see that it is a generalized contact structure. Moreover, using suitable parameter on $E_\pm$ and modifying $\phi$ accordingly on $V^\perp$, one may obtain a family of such structures.
\medskip

Since $\phi,E_\pm$ and the basis of $V^\perp$ are translation invariant, the pair $(\Phi,E_\pm)$ descends to a generalized almost contact structure on the $3$-torus $\T^3=\R^3/\Z^3$ which is not of Poon-Wade type.
\end{example}
\begin{example}(New generalized almost contact structures on Heisenberg group)\label{new eq1}
Let $H_3$ denotes the three dimensional Heisenberg group. We choose a basis $\{X_1,X_2,X_3\}$ for its Lie algebra $\mathfrak{h}_3$ so that $[X_1,X_2]=-X_3$. Let $\{\alpha_1,\alpha_2,\alpha_3\}$ be the corresponding dual basis. Then $d\alpha_3=\alpha_1\wedge\alpha_2$. Consider  $$E_+:=X_1+\alpha_2\,\,\,\text{and}\,\,\,E_-:=\frac{1}{2}(X_2+\alpha_1)\,.$$
In a similar manner, as in Example \ref{new eq}, define an endomorphism $\Phi$ on $TH_3\oplus T^*H_3$ by 
\begin{align*}
&\Phi(E_\pm)=0\,,\Phi(X_1-\alpha_2)=\sqrt{2}\alpha_3\,,\,\Phi(X_3)=\frac{1}{\sqrt{2}}(X_2-\alpha_1)\,;\\
&\Phi(X_2-\alpha_1)=-\sqrt{2} X_3\,,\,\,\text{and}\,\,\Phi(\alpha_3)=-\frac{1}{\sqrt{2}}(X_1-\alpha_2)\,.
\end{align*}
So, $(\Phi,E_\pm)$ is a generalized almost contact structure on $H_3$. The corresponding $L_\pm$ (\eqref{contact dirac}) is of the form 
\begin{align*}
    L_\pm=\spn\left\{E_\pm,X_1-\alpha_2-\sqrt{2}i\alpha_3,X_2-\alpha_1+\sqrt{2}iX_3\right\}
\end{align*}
Since the Courant bracket between $X_3,\alpha_1,$ and $\alpha_2$ and any element among $X_1,X_2,X_3,\alpha_1,\alpha_2,$ and $\alpha_3$ are equal to zero, we have $$[E_+,X_2-\alpha_1+\sqrt{2}iX_3]=-X_3\,,\,\,\,\text{and}\,\,\,[E_-,X_1-\alpha_2-\sqrt{2}i\alpha_3]=\frac{1}{2}X_3+\frac{i}{\sqrt{2}}\alpha_1\,.$$ It follows that both the subbundles $L_+,L_-$ are not Courant involutive. Therefore, $(\Phi,E_\pm)$ does not define a generalized contact structure.
\end{example}
\section{Boothby-Wang construction}\label{sec3}
In this section, we provide a complete description of the Boothby-Wang construction for any generalized contact manifold $M$. Under mild conditions, we give a principal bundle structures on $M$ and establish a criteria for admitting generalized complex structures on its leaf spaces.
\medskip

Let $(M,\Phi,E_{\pm})$ be a generalized almost contact manifold. Then we can write $\Phi$ in the following form:
\begin{equation}\label{Phi}
\Phi=\begin{pmatrix}
    \phi & \pi\\
    \theta & -\phi^*
\end{pmatrix}:TM\oplus T^{*}M\longrightarrow TM\oplus T^{*}M\,,    
\end{equation}
where $\phi\in\en(TM)$, $\theta\in\Hom_{C^\infty(M)}(TM,T^*M)$ and $\pi\in\Hom_{C^\infty(M)}(T^*M,TM)\,.$ Using $\Phi^*=-\Phi$ (cf. Definition \ref{cntct def}), we get $\theta\in\Omega^{2}(M)$ and $\pi\in C^\infty(\wedge^{2}TM)$. Here $\Phi^*,\phi^*$ are dual maps of $\Phi,\phi$, respectively, as defined in Remark \ref{rmk notatn}. Set
\begin{equation}\label{f-eta}
\begin{aligned}
    &R:=X_++X_-\quad\text{and}\quad \eta:=\eta_++\eta_-\,,\\
    &R':=X_--X_+\quad\text{and}\quad \eta':=\eta_+-\eta_-\,,
    \quad\text{where\, $E_{\pm}=X_{\pm}+\eta_{\pm}$}\,.\\
\end{aligned}
\end{equation}
By Definition \ref{cntct def}, it follows immediately that $\eta(R)=1=\eta'(R')$. Clearly, we obtain a canonical (integrable) foliation $\mathcal{R}$ (respectively, $\mathcal{R'}$) generated by the nowhere vanishing vector field $R$ (respectively, $R'$), and denote the corresponding leaf space by $M/\mathcal{R}$ (respectively, $M/\mathcal{R'}$). This follows because $\langle R\rangle\subset TM$ is a rank-$1$ distribution, and every one-dimensional distribution is integrable.

\begin{definition}\label{def:reg}
    A generalized (almost) contact structure $(\Phi,E_{\pm})$ on $M$ is called \textit{complete} if at least one of the associated vector fields $R$ or $R'$ (cf. \eqref{f-eta}), is complete. We call $(M,\Phi,E_{\pm})$ \textit{regular} if at least one of the foliations $\mathcal{R}$ or $\mathcal{R'}$ of $M$ into $R$-orbits or $R'$-orbits, respectively, is simple, i.e., the leaf space $M/\mathcal{R}$ or $M/\mathcal{R'}$ is a smooth manifold such that the corresponding canonical projection 
    \begin{equation}\label{psi}
      \Psi:M\longrightarrow M/\mathcal{R}\,\,\,(\text{respectively,}\,\,\,\Psi':M\longrightarrow M/\mathcal{R'})  
    \end{equation}
    is a surjective submersion, i.e., $\Psi$ (respectively, $\Psi'$) is a smooth fibration.
\end{definition}
\begin{remark}
    Since, $\mathcal{L}_{X_+}X_-=\mathcal{L}_{R}R'$, it follows that $R$ and $R'$ are complete vector fields if and only if $X_+$ and $X_-$ are, provided $[E_+,E_-]=0$.
\end{remark}

\subsection{Induced Principal bundle structures}\label{princi sec}
Let $(M^{2n+1},\Phi,E_{\pm})$ be a connected regular generalized almost contact manifold with the canonical foliation $\mathcal{R}$, induced by the associated vector field $R$ (cf. \eqref{f-eta}). It follows that the fibers of the fibration $M \xrightarrow{\Psi} M/\mathcal{R}$ are orbits of the nowhere vanishing vector field $R$ which are closed submanifolds in this case, and are diffeomorphic either to the circle $S^1$ ($\cong U(1)$) or to $\R$. On a compact $R$-orbit $M_q:=\Psi^{-1}(q)$ for $q\in M/\mathcal{R}$, the flow line, generated by $R$ is periodic with the minimal period $\psi(q)$. In particular, if $M$ is compact, then all $R$-orbits are circles, and in this case, we obtain the smooth function 
\begin{equation}
    \tilde{\psi}:M\to\R^+\,\,,\,y\mapsto\tilde{\psi}(y)\,
\end{equation} defined by assigning to each point $y\in M$ the period $\tilde{\psi}(y)$ of the flow line of the vector field $R$ through $y$; see \cite[Lemma 7.2.6]{gei08}. It follows that $\inf_{y\in M_q}(\tilde{\psi}(y))=\psi(q)$ for all $q\in M/\mathcal{R}$.
\medskip

Now assume that the fibration $\Psi:M \longrightarrow M/\mathcal{R}$ is a $S^1$-bundle and that $i_Fd\eta=0$. Consider a connected local trivialization $V\subset M/\mathcal{R}$ with coordinate $(x_i)\in\R^{2n}$. 
Using the standard covering map of $S^1$, $$\R\ni t\to[t]\in\R/\Z\cong S^1\,,$$
we can view functions on $V\times S^1$ as functions on $V\times\R$ that are $1$-periodic in the $\tau$-variable, where $(x_i,\tau)$ denotes the coordinates on $V\times\R$ lifted from $V\times S^1$. With this identification, the vector field $R$ and the $1$-form $\eta$ lift naturally to $V\times\R$, and still satisfy 
\begin{equation}\label{imp eq}
    \eta(R)=1,\quad i_Fd\eta=0\,.
\end{equation}
Let us write  $\eta$ and $R$ in coordinates $(x_i,\tau)$ as
\begin{equation}\label{eq1}
    \eta=f\,d\tau+\sum^{2n}_{i=1}g_i\,dx_i\,,\quad R=\frac{1}{f}\,\frac{\partial}{\partial\tau}\,,
\end{equation}
where $f$ and $g_i$'s are smooth functions on $V\times\R$, $1$-periodic in $\tau$. This is possible since $\eta(R)=1$, which implies that $f$ is nowhere vanishing. Note that
\begin{equation}\label{period}
    \psi(x)=\int^1_0f(x,\tau)\,d\tau
\end{equation}
is the minimal period of $R$ on $M_x$ for $x\in V$. From \eqref{eq1}, we get $$d\eta=\sum_i\left(\frac{\partial f}{\partial x_i}dx_i\wedge d\tau+\frac{\partial g_i}{\partial\tau}d\tau\wedge dx_i\right)+\sum^{2n}_{i,j=1}\frac{\partial g_i}{\partial x_j}dx_j\wedge dx_i\,.$$
Because $i_Rd\eta=0$, this implies that, for all $i$,

    $$\frac{1}{f}\left(\frac{\partial f}{\partial x_i}-\frac{\partial g_i}{\partial\tau}\right)dx_i=0\,,
    \implies\frac{\partial f}{\partial x_i}=\frac{\partial g_i}{\partial\tau}\,.$$
So we have \begin{equation}\label{d-eta}
    d\eta=\sum^{2n}_{i,j=1}\frac{\partial g_i}{\partial x_j}dx_j\wedge dx_i\,.
\end{equation}
Consequently, for $x=(x_i)$,
$$\frac{\partial\psi(x)}{\partial x_i}=\int_0^1\frac{\partial f(x,\tau)}{\partial x_i}\,d\tau=\int_0^1\frac{\partial g_i(x,\tau)}{\partial\tau}\,d\tau=g_i(x,1)-g_i(x,0)=0\,,$$ as $g_i$'s are $1$-periodic in $\tau$, and since $M$ is connected, $\gamma:=\psi(x)$ is constant for all $x\in M/\mathcal{R}$. It follows that $e^{tR}=id$ if and only if $t\in\gamma\cdot\Z$. In other words, $\Phi:M\longrightarrow M/\mathcal{R}$ is a principal $S^1$-bundle where $S^1$ is diffeomorphic to the quotient space $\R/\gamma\cdot\Z$.
\medskip

Let us change the coordinates in $V\times\R$ into $(x_i,t)$, parametrizing fibers of $V\times\R$ by trajectories of $\frac{1}{f}\,\frac{\partial}{\partial\tau}$ (cf. \eqref{eq1}), and get the following diffeomorphism.
\begin{equation}\label{para}
    \phi: V\times\R\longrightarrow V\times\R,\,\, \phi(x,\tau)=(x,t)\,\,\forall\,\,(x,\tau)\in V\times\R\,,
\end{equation} where $t:=t(x,\tau)=\int_0^\tau f(x,s)\,ds$. Then, by applying the diffeomorphism in \eqref{para}, the expressions for $\eta$ and $R$ can be rewritten in the coordinates $(x_i,t)$ as follows.\begin{equation}\label{eq2}
 R=\frac{\partial}{\partial t},\quad\eta=dt+\sum_ih_i\,dx_i\,,
\end{equation} for some smooth functions $h_i:V\times\R\longrightarrow\R$. Applying $i_Rd\eta=0$ again, we can see that for each $i$, $\frac{\partial h_i}{\partial t}=0$, implying $h_i=\Phi^*h'_i$ on $V\times\R$ for some smooth functions $h'_i:V\longrightarrow\R$ and so, we get
$$d\eta=\sum_{i,j}\frac{\partial h_i}{\partial x_j}(x)\,dx_j\wedge dx_i\,.$$ Thus, there exist an unique closed $2$-form $\omega_V\in\Omega^2(V)$ such that $d\eta=\Phi^*\omega_V$ on $V\times S^1$. In particular, $\omega_V=\sum_{i,j}\frac{\partial h'_i}{\partial x_j}(x)\,dx_j\wedge dx_i=d(\sum_i h'_i\,dx_i)$ on $V$. Since $V\times S^1$ is an arbitrary local trivialization of $\Phi$, globally there exist an unique closed $2$-form $\omega\in\Omega^2(M/\mathcal{R})$ such that $\omega|_V=\omega_V$ on any local trivialization $V$. Therefore, we have proved the following.
\begin{theorem}\label{thm1}
Let $(M,\Phi,E_{\pm})$ be a connected, regular generalized almost contact manifold with associated vector field $R$ and $1$-form $\eta$ as in \eqref{f-eta}, and suppose $i_Rd\eta=0$. Denote by $\mathcal{R}$ the one-dimensional foliation generated by $R$, and let $M/\mathcal{R}$ be the leaf space of $\mathcal{R}$. Assume the canonical projection $\Psi:M\longrightarrow M/\mathcal{R}$ (cf. \eqref{psi}) is a smooth fiber bundle whose typical fiber is $S^1$. Then
\vspace{0.2em}
\begin{enumerate}
\setlength\itemsep{0.5em}
    \item The flow generated by $R$ on $M$ is periodic, and all $R$-orbits have the same minimal period $\gamma$.
    \item The fibration $\Psi:M\longrightarrow M/\mathcal{R}$ yields a principal $S^1$-bundle structure on $M$, with $\eta$ as its connection $1$-form and curvature the closed $2$-form $\omega\in\Omega^2(M/\mathcal{R})$ such that $d\eta=\Phi^*\omega$. Here $S^1\cong \R/\gamma\cdot\Z\,.$
\end{enumerate}
\end{theorem}
\begin{remark}
 The main importance of the condition $i_R d\eta = 0$ is that  $\mathcal{L}_R\eta=0$, that is the $1$-form $\eta$ is $\mathcal{R}$-invariant. As a result, $d\eta$ depends only on directions transverse to $\mathcal{R}$, and hence is the pullback of a $2$-form on the quotient $M/ \mathcal{R}$.
\end{remark}
\begin{prop}\label{prop1}
    Given the assumptions stated in Theorem \ref{thm1}, the Euler class of the principal $S^1$-bundle $M\xrightarrow{\Psi}M/\mathcal{R}$ is represented by $[\omega/\gamma]\in H^2(M/\mathcal{R},\Z)\subset H^2(M/\mathcal{R},\R)$ where $\gamma$ denotes the minimal period of the foliation $\mathcal{R}$.
\end{prop}
\begin{proof}
   Let $\{V_{\alpha},\phi_{\alpha}\}$ be a family of local trivializations 
    $\phi_{\alpha}:\Phi^{-1}(V_{\alpha})\longrightarrow V_{\alpha}\times S^1\,,$ with transition functions 
$\phi_{\alpha\beta}:V_{\alpha}\cap V_{\beta}\longrightarrow S^1\cong\R/\gamma\cdot\Z\,,$ such that $\{V_{\alpha}\}$ is a good cover of $M/\mathcal{R}$. Because the transition functions satisfy the cocycle condition, it implies that, for $x\in V_{\alpha}\cap V_{\beta}\cap V_{\delta}$,
$$\phi_{\alpha\beta}(x)+\phi_{\beta\delta}(x)+\phi_{\delta\alpha}(x)\equiv 0\mod(\gamma\cdot\Z)\,.$$ Now on $V_\alpha\times S^1$, we have 
$(\phi_\alpha^{-1})^{*}(\eta|_{\Phi^{-1}(V_{\alpha})})=dt_\alpha+\theta_\alpha$ (cf. \eqref{eq2}) where $d\theta_\alpha=\omega|_{V_\alpha}$.
Then by straightforward modification of \cite[Proposition 4.1]{grabo25} and because $\theta_\beta-\theta_\alpha=d\phi_{\alpha\beta}$, we get that $[\omega/\gamma]\in H^2(M/\mathcal{R},\Z)$.

\end{proof}
\begin{remark}
Note that for $\eta'_{\pm}=\pm(\eta_+-\eta_-)$ and $R'_{\pm}=\pm(R_--R_+)$, we have $\eta'_{\pm}(R'_{\pm})=1$. Under the assumptions of Theorem \ref{thm1}, this yields two distinct connection forms, $\eta'_+$ and $\eta'_-$, whose curvatures may differ; however, their curvature class is the same. In other words, the de Rham cohomology class of the curvature of $\eta$ in Theorem \ref{thm1} is independent of the choice of the connection form.
\end{remark}
\begin{remark}\label{rmk1}
    Observe that Theorem \ref{thm1} also holds when $M$ is an arbitrary smooth manifold, provided there exist a $1$-form $\eta'$ and a vector field $R'$ with $\eta'(R')\in\R\backslash\{0\}$, and the hypotheses of Theorem \ref{thm1} are satisfied for this pair.
\end{remark}
Given a regular generalized almost contact manifold $(M^{2n+1},\Phi,E_{\pm})$, let us assume that the associated vector field $R$ (cf. \eqref{f-eta}) is complete. Consequently, its flow is defined for all time and yields a smooth action $$(t,x)\to e^{tR}(x),\,\,\,(t,x)\in\R\times M\,,$$ of the additive group $(\R,+)$ on $M$. In general, the $R$-orbits of the flow may include both compact and non-compact orbits.
\medskip

So, suppose that all $R$-orbits are non-compact, and hence diffeomorphic to $\R$. In this case, the $\R$-action generated by $R$ is free, and therefore the projection $\Psi:M\longrightarrow M/\mathcal{R}$ (cf. Definition \ref{def:reg}) is automatically a fiber bundle; see \cite[Corollary 31]{meig}. Moreover, since the free $\R$-action given by the flow preserves the fibers of $\Psi$, the bundle $M\xrightarrow{\Psi} M/\mathcal{R}$ becomes an $\R$-principal bundle. By a direct modification of the proof of Theorem \ref{thm1} and of \cite[Section 5.1]{grabo25}, we get that the $1$ form $\eta$ defines a connection on the bundle, with curvature given by the $2$-form $\omega\in\Omega^2(M/\mathcal{R})$, as in Theorem \ref{thm1}.
\medskip

Now assume that there exist a compact orbit of the fibration $\Psi:M\longrightarrow M/\mathcal{R}$. Let $F=\Psi^{-1}(x)$ ($\cong S^1$) be the compact orbit of $\mathscr{R}$. By the Tubular Neighborhood Theorem, there exists an open (tubular) neighborhood $V\subset M$ of $F$ which is diffeomorphic to the normal bundle $\mathcal{N}_{F}$ of $F$. One can see that $\mathcal{N}_{F}$ is just the pullback of the normal bundle, denoted by $\mathcal{N}_R$, of the foliation $\mathcal{R}$ via the inclusion map $F\hookrightarrow M$. Consider the \emph{Bott connection} (cf. \cite[Section 6]{bott71}) on $\mathcal{N}_R$ which is flat along the leaves of $\mathcal{R}$. Then its pullback on $\mathcal{N}_{F}$ gives a flat connection. Thus, by \cite[Proposition 1.2.5]{kobayashi14}
$$\mathcal{N}_{F}\cong\tilde{F}\times_{\Gamma}\R^{2n}$$ where $\Gamma:\pi_1(F)\longrightarrow GL_{2n}(\R)$ is the linear holonomy representation of $\pi_1(F)$ and $\tilde{F}$ is the universal cover of $F$. Note that, since $F$ is a compact embedded submanifold, the holonomy group of $F$ is finite and exactly the image of $\Gamma$. Since $M/\mathcal{R}$ is a manifold and $H^1(\tilde{F},\R)=0$, it follows from \cite[Theorem 2]{thurston} and \cite[Section 5.2]{grabo25} that there exists a connected open neighborhood $V\subset M/\mathcal{R}$ such that $\Psi^{-1}(V)\xrightarrow{\Psi} V$ is a trivializable fiber bundle over $V$ with the typical fiber $S^1$. This is followed by the fact that the holonomy is trivial in this case.  Assuming $i_Rd\eta=0$ with $R$ and $\eta$ as in \eqref{f-eta}, Theorem \ref{thm1} implies that $\Psi^{-1}(V)\xrightarrow{\Psi} V$ is a principal $S^1$-bundle. Let us fix such $V$ and denote the associated minimal period  by $\gamma$. Consider the set 
$$(M/\mathcal{R})_{\gamma}:=\left\{x\in M/\mathcal{R}\,|\,\psi(x)=\gamma\,\,\text{and $\Psi^{-1}(x)$ is a compact orbit}\right\}\,,$$ where $\psi(x)$ is the minimal period on $\Psi^{-1}(x)$, as in \eqref{period}. Clearly, $M_{\gamma}$ is an open set. Let $x'\in M/\mathcal{R}$ be any point in the closure of $(M/\mathcal{R})_\gamma$, and choose $y'\in\Psi^{-1}(x')$. Since the submersion $\Psi$ is an open map, there exists, in a neighborhood of $y'$, a sequence $\{y_n'\}$ such that $$\Psi(y_n')=x'_n\in(M/\mathcal{R})_\gamma\quad\text{and}\quad y'_n\to y'\,\,\text{as}\,\,n\to\infty\,.$$ Because the vector field $R$ is complete, its flow $\psi_t:=e^{tR}$ is globally defined. For each $n\in\N$, the point $y'_n$ lies on a periodic orbit of period $\gamma$. Hence $$y'_n=\psi_\gamma(y'_n)\,.$$ Taking limits and using continuity of the flow yields
$$\psi_\gamma(y')=\lim_{n\to\infty} \psi_\gamma(y_n')
=\lim_{n\to\infty} y_n'= y'\,.$$ Thus, the fiber $\Psi^{-1}(x')$ is itself a periodic orbit of period $\gamma$. Moreover, this period is minimal, since the minimal period function \eqref{period} is locally constant on periodic orbits. Consequently, if $M/\mathcal{R}$ is connected and $R$ admits one periodic orbit of minimal period $\gamma$, then every orbit is periodic with the same minimal period $\gamma$. In particular, $(M/\mathcal{R})_{\gamma}$ is always closed, and when $M/\mathcal{R}$ is connected, we have $$(M/\mathcal{R})_{\gamma}=M/\mathcal{R}\,.$$ Observe that $[R,\eta]=\mathcal{L}_R\eta=i_Rd\eta$ (cf. \eqref{bracket}). 
\begin{theorem}\label{thm2}
    Let $(M,\Phi,E_{\pm})$ be a regular generalized almost contact manifold with complete vector field $R$ and $1$-form $\eta$ as in \eqref{f-eta}, and assume that $[R,\eta]=0$. Then,
    \vspace{0.2em}
    \begin{enumerate}
    \setlength\itemsep{0.5em}
    \item For the vector field $R$, all of its orbits are diffeomorphic either to $\R$ or to $S^1$. In other words, the set of orbits cannot contain both compact and non-compact orbits simultaneously.
     \item The canonical projection $\Psi:M\longrightarrow M/\mathcal{R}$ (cf. \eqref{psi}) endows $M$ with the structure of a principal $\R$- or $U(1)$-bundle, for which $\eta$ acts as a connection $1$-form, and its curvature is the closed $2$-form $\omega\in\Omega^2(M/\mathcal{R})$ determined by $d\eta=\Phi^*\omega$.
        \item If $\gamma\in(0,\infty]$ is the minimal period of the globally defined flow $e^{tR}$, then the cohomology class $[\omega/\gamma]\in H^2(M/\mathcal{R},\Z)\subset H^2(M/\mathcal{R},\R)$ represents the Euler class of the bundle.
    \end{enumerate}
\end{theorem}
\begin{proof}
Follows from the preceding discussion and Proposition \ref{prop1}.   
\end{proof}
\begin{remark}
    Note that, the principal $S^1$-bundles in Theorem \ref{thm2} (also in Theorem \ref{thm1}), are classified by their Euler classes in integral cohomology $H^2(M/\mathcal{R},\Z)$.
\end{remark}
\begin{remark}\label{rmk2}
   As in Remark \ref{rmk1}, Theorem \ref{thm2} extends to any smooth manifold $M$, provided there exist a $1$-form $\eta'$ and a vector field $R'$ such that $\eta'(R')\not\equiv 0$, and the assumptions of Theorem \ref{thm2} hold for this pair.
\end{remark}

\subsection{Induced generalized complex structures}
Let $(M^{2n+1},\Phi,E_{\pm})$ be a (connected) generalized almost contact manifold, together with the nowhere vanishing pair $\{R,\eta\}$ (cf. \eqref{f-eta}) and the induced foliation $\mathcal{R}$. Let $T\mathcal{R}:=\langle R\rangle$ be the corresponding involutive subbundle of $TM$ of rank $1$, called the tangent bundle of the foliation. The normal bundle of the foliation, denoted by $\mathcal{N}_R$, is defined by
\begin{equation}\label{nrml bndl}
    \mathcal{N}_R:=TM/T\mathcal{R}\,.
\end{equation}
As the exact sequence 
\[\begin{tikzcd}[ampersand replacement=\&]
	0 \& {T\mathcal{R}} \& TM \& {\mathcal{N}_R} \& 0
	\arrow[from=1-2, to=1-3]
	\arrow[from=1-3, to=1-4]
	\arrow[from=1-1, to=1-2]
	\arrow[from=1-4, to=1-5]
\end{tikzcd}\,\]
splits smoothly,
$\mathcal{N}_R$ may be regarded as a subbundle of $TM$ (after fixing a Riemannian metric), complementary to $T\mathcal{R}$, and has rank $2n$. Let $\ann(R):=\{\beta\in T^*M\,|\,\beta(R)=0\}\subset T^*M$ be the annihilator of $R$, and consider the subbundle on $M$
\begin{equation}\label{subbndl}
\mathbb{T}_{\mathcal{R}}M:=\ann(R)\oplus\mathcal{N}_R\subset TM\oplus T^*M\,.
\end{equation}
\begin{definition}\label{imp def}
~
    \begin{enumerate}
        \item Given a subset $K\subseteq TM\oplus T^{*}M$, the orthogonal complement of $K$, denoted by $K^\perp$, with respect to the bilinear form \eqref{bilinear} is defined as 
        $$K^\perp:=\left\{s\in TM\oplus T^{*}M\,|\,\langle s,K\rangle=0\right\}\,.$$
        \item A section $s\in C^\infty(TM\oplus T^{*}M)$ is called $R$-invariant if $[C^{\infty}(T\mathcal{R}),s]\subseteq C^{\infty}(T\mathcal{R})$. Denote by $C^{\infty}_{\mathcal{R}}(K)$, the set of $R$-invariant sections of a subset $K\subseteq TM\oplus T^{*}M$. Additionally, we say $R$ preserves $K$ if $[C^{\infty}(T\mathcal{R}),C^{\infty}(K)]\subseteq C^{\infty}(K)$.
    \end{enumerate}
\end{definition}
Observe that $\mathbb{T}_{\mathcal{R}}M\cong TM\oplus\ann(R)/ T\mathcal{R}\oplus\{0\}$. Consequently, the bilinear pairing \eqref{bilinear} descend naturally to $\mathbb{T}_{\mathcal{R}}M$. However, this bundle may not be closed under the Courant bracket (\eqref{bracket}) for arbitrary sections. When restricted to basic sections, that is, sections constant along the leaves, $\mathbb{T}_{\mathcal{R}}M$ carries a natural Courant algebroid structure induced from $TM\oplus T^*M$. More precisely, $C^{\infty}_{\mathcal{R}}(\mathbb{T}_{\mathcal{R}}M)$ is closed under the Courant bracket, in the following way. 
\vspace{0.3em}

Notice that $[R,s]\in C^\infty(TM\oplus\ann(R))$ for any $s\in C^\infty(TM\oplus\ann(R))$, implying that 
\begin{equation}\label{1.1}
[C^\infty(T\mathcal{R})),C^\infty(TM\oplus\ann(R))]\subseteq C^\infty(TM\oplus\ann(R))\,.   
\end{equation}
Now, let $s_1,s_2\in C_{\mathcal{R}}^\infty(TM\oplus\ann(R))$ (cf. Definition \ref{imp def}). By \eqref{eqcc}, we obtain
\begin{equation}\label{1}
0=\pr_{TM}(s_1)\langle s_2,R\rangle=\langle [s_1,s_2],R\rangle +\langle s_2,[s_1,R]\rangle=\langle [s_1,s_2],R\rangle\,,    
\end{equation}
and also,
\begin{equation}\label{2}
\begin{aligned}
    &[[R, s_1], s_2] + [s_1, [R, s_2]]-[R, [s_1, s_2]]\\
    &=\frac{1}{3}d\bigg(\langle [R,s_1],s_2\rangle+\langle R,[s_1,s_2]\rangle-\langle s_1,[R,s_2]\rangle\bigg)\\
    &=\frac{1}{3}d(\langle R,[s_1,s_2]\rangle)\quad(\text{as $[R,s_i]\in C^\infty(T\mathcal{R})$ for $i=1,2$})\\
    &=0\quad(\text{by \eqref{1}})\,.
\end{aligned}
\end{equation}
Since $[R,s_1],[R,s_2]\in C^\infty(T\mathcal{R})$, the sum $[[R, s_1], s_2] + [s_1, [R, s_2]]$ is in $C^\infty(T\mathcal{R})$. Consequently, using \eqref{2}, we get that $[R,[s_1, s_2]]\in C^\infty(T\mathcal{R})\,.$ Therefore
$$[C_{\mathcal{R}}^\infty(TM\oplus\ann(R)),C_{\mathcal{R}}^\infty(TM\oplus\ann(R))]\subseteq C_{\mathcal{R}}^\infty(TM\oplus\ann(R))\,.$$
Using \eqref{1.1}, we conclude that  $C_{\mathcal{R}}^\infty(\mathbb{T}_{\mathcal{R}}M)$ is Courant involutive, that is,
\begin{equation}\label{1.2}
[C_{\mathcal{R}}^\infty(\mathbb{T}_{\mathcal{R}}M),C_{\mathcal{R}}^\infty(\mathbb{T}_{\mathcal{R}}M)]\subseteq C_{\mathcal{R}}^\infty(\mathbb{T}_{\mathcal{R}}M)\,.
\end{equation}
In other words, $\mathbb{T}_{\mathcal{R}}M$ admits a Courant algebroid structure in the transverse direction of the foliation $\mathcal{R}$. This is possible because $\mathbb{T}_{\mathcal{R}}M$ is a foliated (or basic) vector bundle over $M$, whose transition functions are basic, meaning they are constant along the leaves of the foliation. 
\begin{definition}\label{td def}
 We will refer to the bundle $\mathbb{T}_{\mathcal{R}}M$ as the \textit{transverse Courant algebroid} over $M$, associated to the foliation $\mathcal{R}$. A subbundle $L\subset\mathbb{T}_{\mathcal{R}}M$ of rank $\frac{\dim M-1}{2}$ is called a \textit{transverse Dirac structure} if $L$ is isotropic and $C_{\mathcal{R}}^\infty(L)$ is closed under the Courant bracket.
\end{definition}

Set
\begin{equation}\label{imp not}
\begin{aligned}
    &T\mathcal{R}^{\perp}:=\left\{s\in TM\oplus T^{*}M\,|\,\langle s,R\rangle=0\right\}=TM\oplus\ann(R)\,;\\
    &C^{\infty}_{\mathcal{R}}(T\mathcal{R}^{\perp}):=\left\{s\in C^{\infty}(T\mathcal{R}^{\perp})\,|\,[C^{\infty}(T\mathcal{R}),s]\subset C^{\infty}(T\mathcal{R})\right\}\,.
\end{aligned}
\end{equation}
\begin{lemma}\label{lemma 1}
   Let $N^n$ be a smooth manifold admitting a nowhere-vanishing vector field $X$. Let $L\subset TN\oplus T^{*}N$ be a maximal isotropic subbundle such that $X\not\in C^\infty(L)$. Denote by $L_X\subset TN$ the line subbundle spanned by the vector field $X$. Then 
   \begin{enumerate}
   \setlength\itemsep{0.5em}
       \item $H:=L\cap(TN\oplus\ann(X))$ is an isotropic subbundle of $TN\oplus T^{*}N$.
       \item $H^\perp:=\left\{s\in TN\oplus T^{*}N\,|\,\langle s,H\rangle=0\right\}=L\oplus L_X$.
   \end{enumerate}
\end{lemma}
\begin{proof}
~
\begin{enumerate}
 \setlength\itemsep{0.3em}
    \item Consider the bundle map $\varphi:L\longrightarrow N\times\R\cong L_X$ defined by $\varphi(l)=\langle X,l\rangle$. Since $X\not\in C^\infty(L)$ and $L$ is maximal isotropic bundle, the map $\varphi$ has rank $1$ at every point of $N$. Consequently, its kernel $\ker(\varphi)=H$ has constant rank $n-1$. Therefore $H$ is a subbundle, and by construction, it is isotropic.
    \item Fix any $x\in N$. Now $\dim L_x=n$ and $\dim H_x=n-1$. It follows that 
    \begin{align*}
        &\dim(L\oplus L_X)_x=\dim L_x+1=n+1\,;\\
        &\dim H^\perp_x=\dim(T_xN\oplus T_x^*N)-\dim H_x=n+1\,.
    \end{align*}
Let $u=l+rX\in L_x\oplus (L_X)_x$ where $r\in\R$, and $l\in L_x$. Then, for any $f\in H_x$,
$$\langle u,f\rangle=\langle l,f\rangle+r\langle X,f\rangle=0\,,$$ since $f\in T_xN\oplus\ann(X)_x$\,, and $L_x$ is isotropic. So $$L_x\oplus (L_X)_x\subset H^\perp_x\,.$$ Since $\dim(L\oplus L_X)_x=\dim H^\perp_x$, we have $L_x\oplus (L_X)_x=H^\perp_x$. Thus $L\oplus L_X=H^\perp$.
\end{enumerate}
\end{proof}
\begin{prop}\label{prop c}
~
    \begin{enumerate}
    \setlength\itemsep{0.5em} 
    \item $TM\oplus T^{*}M=\img\Phi \oplus\ker\Phi=T\mathcal{R}^{\perp}\oplus L_\eta$ where $L_\eta$ is the real line bundle generated by $\eta$.
        \item $\rank(T\mathcal{R}^{\perp})=4n+1$, $\rank(\mathbb{T}_{\mathcal{R}}M)=4n$ and $\rank(\ann(R))=2n$.
        \item $H_\pm:=L_\pm\cap (T\mathcal{R}^{\perp}\otimes\C)=\left\{X+\xi\in L_\pm\,|\,\xi(R)=0\right\}$ are isotropic complex subbundles of $(TM\oplus T^{*}M)\otimes\C$ where $L_\pm\subset (TM\oplus T^{*}M)\otimes\C$ are maximal isotropic subbundles, defined as in \eqref{contact dirac}. Moreover,
        \[
        \rank(H_\pm)=
        \begin{cases}
           2n & \text{if $R\not\in C^\infty(L_\pm)$}\,;\\
           2n+1 & \text{if $R\in C^\infty(L_\pm)$}\,.
        \end{cases}
         \] 
         \item $H_\pm^\perp:=\left\{s\in TM\oplus T^{*}M\,|\,\langle s,H_\pm\rangle=0\right\}=L_\pm+(T\mathcal{R}\otimes\C)\,.$
         \item $\left(L_\pm+(T\mathcal{R}\otimes\C)\right)\cap (T\mathcal{R}^{\perp}\otimes\C)=H_\pm+(T\mathcal{R}\otimes\C)\,.$ Moreover, the subbundle $H_\pm+(T\mathcal{R}\otimes\C)$ is isotropic and the rank of $H_\pm+(T\mathcal{R}\otimes\C)$ is $2n+1$.
    \end{enumerate}
\end{prop}
\begin{proof}
    $(1)$ and $(2)$ follow from direct computation together with the rank-nullity theorem for finite-dimensional vector spaces.\\\\
    $(3)$ Let $R\in C^{\infty}(L_\pm)$. Then
    \begin{equation}\label{eq1.5}
        T\mathcal{R}\otimes\C\subseteq L_\pm\cap\overline{L_\pm}=L_{E_\pm}\otimes\C\,,
    \end{equation}
    and by comparing vector space dimensions we obtain $T\mathcal{R}=L_{E_\pm}$. This holds if and only if $(\Phi,E_{\pm})$ is a generalized almost contact structure of Poon-Wade type. Consequently, we have $H_\pm=L_\pm\,,$ since $T\mathcal{R}^{\perp}\otimes\C=(L_{E_\pm}\otimes\C)\oplus E^{(1,0)}\oplus E^{(0,1)}$.
    \vspace{0.5em}
    
    Now, let $R\not\in C^{\infty}(L_\pm)$. Consider the projection map $\pr_{L_{\eta}}:(TM\oplus T^{*}M)\otimes\C\longrightarrow L_\eta\otimes\C$ and restrict it to $L_\pm$. Then $$\ker(\pr_{L_{\eta}}|_{L_\pm})=L_\pm\cap (T\mathcal{R}^{\perp}\otimes\C)\,.$$ The rank of $\pr_{L_{\eta}}|_{L_\pm}$ is either $0$ or $1$. Suppose that at some point $x\in M$, $\rank(\pr_{L_{\eta}}|_{L_\pm})=0$. Then $L_\pm\subseteq (T\mathcal{R}^{\perp}\otimes\C)$ at $x$. Note that $(T\mathcal{R}\otimes\C)\cap L_\pm=\{0\}$ as $\langle L_\pm,E_\pm\rangle=0$. Consequently, $(L_\pm)_x\oplus(T_x\mathcal{R}\otimes\C)$ is an isotopic subspace in $(T\mathcal{R}^{\perp})_x\otimes\C$, and contains $(L_\pm)_x$. This is impossible, because $(L_\pm)_x$ are maximal isotropic subspaces of $(T_xM\oplus T^*_xM)\otimes\C$, which would imply $(L_\pm)_x=(L_\pm)_x\oplus(T_x\mathcal{R}\otimes\C)$, and therefore $R\in(L_\pm)_x$. Hence $L_\pm\cap (T\mathcal{R}^{\perp}\otimes\C)$ have constant complex rank $2n$, implying they are complex subbundles of $(TM\oplus T^{*}M)\otimes\C$.\\\\
    $(4)$ and $(5)$ Follow from a straightforward modification of Lemma \ref{lemma 1} and and direct computation, respectively. 
\end{proof}
The proof of the next Proposition follows similar arguments in \cite[Lemma 4.1]{zambon1} and provides a slight generalization, in terms of complexified version.
\begin{prop}\label{prop  c1}
The sets of $R$-invariant section of $H_\pm+(T\mathcal{R}\otimes\C)$ are closed under the Courant bracket, that is,
$$[C^\infty_{\mathcal{R}}(H_\pm+(T\mathcal{R}\otimes\C)),C^\infty_{\mathcal{R}}(H_\pm+(T\mathcal{R}\otimes\C))]\subseteq C^\infty_{\mathcal{R}}(H_\pm+(T\mathcal{R}\otimes\C))\,,$$ provided, $L_\pm$ are also Courant involutive, respectively. In particular, $[C^\infty_{\mathcal{R}}(H_\pm),C^\infty_{\mathcal{R}}(H_\pm)]\subseteq C^\infty(H^\perp_\pm)$.
\end{prop}

\begin{proof}
Let $s_1,s_2\in C^\infty_{\mathcal{R}}(H_\pm+(T\mathcal{R}\otimes\C))$. Then a straightforward modification of equations \eqref{1} and \eqref{2}, together with the fact that $[R,s_i]\in C^\infty(T\mathcal{R}\otimes\C)$ ($i=1,2$), yields the first result. The particular case is an immediate consequence of the fact that $H_\pm\subseteq H_\pm+(T\mathcal{R}\otimes\C)\subset H^\perp_\pm$; see Proposition \ref{prop c}. 
\end{proof}
From Proposition \ref{prop c}, we see that the subbundles $H_\pm+(T\mathcal{R}\otimes\C)$ are, in fact, maximal isotropic subbundles. Then, Proposition \ref{prop  c1} ensures the Courant involutiveness for $R$-invariant sections when $M$ is a generalized contact manifold. Consequently, it follows that 
\begin{equation}\label{td}
  (H_\bullet+(T\mathcal{R}\otimes\C))/T\mathcal{R}\otimes\C  
\end{equation}
defines a transverse Dirac structure (see Definition \ref{td def}) of $\mathbb{T}_{\mathcal{R}}M$ whenever the corresponding $L_\bullet$ (cf. \eqref{contact dirac}) is closed under the Courant bracket, where $\bullet\in\{+,-\}$. Here, we are considering the inclusion 
\begin{equation}\label{identi}
    (H_\pm+(T\mathcal{R}\otimes\C))/T\mathcal{R}\otimes\C\subset (TM\oplus\ann(R))\otimes\C/ (T\mathcal{R}\otimes\C)\cong \mathbb{T}_{\mathcal{R}}M\otimes\C\,.
\end{equation}
Thus, we have established the following result.
\begin{theorem}\label{c-d}
    Let $(M^{2n+1},\Phi,E_{\pm})$ be a generalized contact manifold with the Courant involutive bundle $L_\bullet$ (as in \eqref{contact dirac}), where $\bullet\in\{+,-\}$. Then $M$ admits a transverse Courant algebroid $\mathbb{T}_{\mathcal{R}}M$ (see \eqref{subbndl}), and, moreover, a transverse Dirac structure $$(H_\bullet+(T\mathcal{R}\otimes\C))/T\mathcal{R}\otimes\C$$ as given in \eqref{td}. In particular, $\rank((H_\bullet+(T\mathcal{R}\otimes\C))/T\mathcal{R}\otimes\C)=2n$.
\end{theorem}
\begin{remark}\label{imp rmk1}
Note that, using \eqref{identi}, we have 
\[
     \mathbb{T}_{\mathcal{R}}M\supset (H_\bullet+(T\mathcal{R}\otimes\C))/T\mathcal{R}\otimes\C\cong
        \begin{cases}
           H_\bullet & \text{if $R\not\in C^{\infty}(L_\bullet)$}\,;\\
           E^{(*,*)} & \text{if $R\in C^{\infty}(L_\bullet)$}\,\,,
        \end{cases}
\] 
where $E^{(*,*)}$ is the eigenbundle of $\Phi$ associated with $L_\bullet$. Now, when $R\not\in C^{\infty}(L_\bullet)$,
\begin{align*}
  H_\bullet\cap\overline{H_\bullet}&=(L_\bullet\cap\overline{L_\bullet})\cap(TM\oplus\ann(R))\otimes\C\\
  &=(L_{E_\bullet}\otimes\C)\cap(TM\oplus\ann(R))\otimes\C\\
  &=\{0\}\,.\quad(\text{as $\langle E_\bullet,R\rangle\neq 0$})
\end{align*}
\end{remark}
Assume that $H_\pm$ are Courant involutive whenever the associated subbundles $L_\pm$ are so, respectively. Let $e\in C^\infty(H_\pm)$ and consider $[e,R]\in C^\infty((TM\oplus T^*M)\otimes\C)$. Consider the natural projection map $\pr_{TM}:(TM \oplus T^{\ast}M) \otimes \C \longrightarrow TM \otimes \C $. Using \eqref{eqcc}, for any $f\in C^\infty(H_\pm)$, we  have
$$\pr_{TM}(e)\langle R,f\rangle=\langle [e,R],f\rangle+\langle R,[e,f]\rangle\,\,\,\text{which implies}\,\,\,\langle [e,R],f\rangle=0\,.$$
This is possible because $H_\pm$ are both isotropic and Courant involutive. It follows that $[e,R]\in C^\infty(H_\pm^\perp)$. Applying \eqref{eqcc} with $s_1=s_2=R$ and $s_3=e$, we obtain $[R,e]\in C^\infty(T\mathcal{R}^{\perp}\otimes\C)$. Thus
\begin{equation}\label{eqc1}
[C^\infty(T\mathcal{R}\otimes\C),C^\infty(H_\pm)]\subseteq C^\infty(H_\pm^\perp)\,\,\,\text{and}\,\,\,[C^\infty(T\mathcal{R}\otimes\C),C^\infty(H_\pm)]\subseteq C^\infty(T\mathcal{R}^{\perp}\otimes\C)\,.
\end{equation} Therefore by Proposition \ref{prop c}, we have
\begin{equation}\label{eqc3}
\begin{aligned}
&[C^\infty(T\mathcal{R}\otimes\C),C^\infty(H_\pm)]\subseteq C^\infty(H_\pm)\,;\\
    [C^\infty(H_\pm+&(T\mathcal{R}\otimes\C)),C^\infty(H_\pm+(T\mathcal{R}\otimes\C))]\subseteq C^\infty(H_\pm+(T\mathcal{R}\otimes\C))\,.\\
\end{aligned}    
\end{equation}
In other words, $H_\pm+(T\mathcal{R}\otimes\C)\subset(TM\oplus T^*M)\otimes\C$ are maximal, isotropic, and Courant involutive subbundles. Moreover, let $e\in C^\infty(H^\perp_\pm)$ and $f\in C^\infty(H_\pm)$. Again using \eqref{eqcc}, we get,
\begin{align*}
0=R\langle e,f\rangle&=\langle [R,e],f\rangle+\langle e,[R,f]\rangle\,;\\
&=\langle [R,e],f\rangle\,.\quad(\text{by \eqref{eqc1}})
\end{align*}
Hence 
\begin{equation}\label{eqc2}    [C^\infty(T\mathcal{R}\otimes\C),C^\infty(H^\perp_\pm)]\subseteq C^\infty(H_\pm^\perp)\,\,\,\text{and so,}\,\,\,[C^\infty(H_\pm^\perp),C^\infty(H^\perp_\pm)]\subseteq C^\infty(H_\pm^\perp)\,.
\end{equation}
\begin{theorem}\label{thm c}
    Let $(M^{2n+1},\Phi,E_{\pm})$ be a generalized contact manifold with the Courant involutive bundle $L_\bullet$ (\eqref{contact dirac}), where $\bullet\in\{+,-\}$. Let $\{R,\eta\}$ be the associated nowhere vanishing pair (see \eqref{f-eta}). Denote by $\mathcal{R}$ the foliation induced by $R$. Let $H_\pm,H^\perp_\pm$ be defined as in Proposition \ref{prop c}. Then the following statements are equivalent.
    \vspace{0.2em}
    \begin{enumerate}
    \setlength\itemsep{0.5em}
        \item $H_\bullet$ is Courant involutive.
        \item $[C^\infty(T\mathcal{R}\otimes\C),C^\infty(H_\bullet)]\subseteq C^\infty(H^\perp_\bullet)$.
        \item $H_\bullet+(T\mathcal{R}\otimes\C)$ is a maximal, isotropic and Courant involutive complex subbundle.
        \item $[C^\infty(T\mathcal{R}\otimes\C),C^\infty(H_\bullet+(T\mathcal{R}\otimes\C))]\subseteq C^\infty(H_\bullet+(T\mathcal{R}\otimes\C))$.
        \item $[C^\infty(T\mathcal{R}\otimes\C),C^\infty(H_\bullet)]\subseteq C^\infty(H_\bullet)$.
    \end{enumerate}
In addition, if any one of the equivalent statements holds, we obtain
\vspace{0.2em}
\begin{itemize}
\setlength\itemsep{0.3em}
    \item[i)] $[C^\infty(T\mathcal{R}\otimes\C),C^\infty(H^\perp_\bullet)]\subseteq C^\infty(H_\bullet^\perp)$.
    \item[ii)] $H_\bullet^\perp$ is closed under the Courant bracket.
\end{itemize}
\end{theorem}
\begin{proof}
    \textbf{$\left((1)\Leftrightarrow (2)\right)$} One direction follows from the preceding discussion (cf. \eqref{eqc1}), and the additional statements also hold in this case; see \eqref{eqc2}. For the converse, assume condition $(2)$. Notice that the bundle $H_\bullet$ is isotropic by Proposition \ref{prop c}. Let $e,f\in C^\infty(H_\bullet)$. Since $e,f\in C^\infty(L_\bullet)$, the section $[e,f]$ belongs to $C^\infty(L_\bullet)$. By \eqref{eqcc},
    $$\langle R,[e,f]\rangle+\langle [e,R],f\rangle=R\langle e,f\rangle=0\,.$$ Since $[e,R]\in C^\infty(H^\perp_\bullet)$, we get that $\langle R,[e,f]\rangle=0$, implying $[e,f]\in C^\infty(TM\oplus\ann(R))$. Therefore, $[e,f]\in C^\infty(H_\bullet)$.
     \medskip
     
    \textbf{$\left((1)\Leftrightarrow (3)\right)$} Assume condition $(1)$. Then condition $(3)$ also follows from the preceding discussion (cf. \eqref{eqc3}). Conversely, suppose condition $(3)$ holds. Let $e,f\in C^\infty(H_\bullet)$. Then $[e,f]=s_1+s_2$ for some $s_1\in C^\infty(H_\bullet)$ and $s_2\in C^\infty(T\mathcal{R}\otimes\C)$. So,
    $$\langle R,[e,f]\rangle=\langle R,s_1\rangle+\langle R,s_2\rangle=0\,,$$ as $s_1,s_2\in C^\infty(TM\oplus\ann(R))$, implying $[e,f]\in C^\infty(TM\oplus\ann(R))$.
    Since $L_\bullet$ is Courant involutive, $[e,f]\in C^\infty(L_\bullet)$ also. Thus $[e,f]\in C^\infty(H_\bullet)$. 
     \medskip
     
    \textbf{$\left((3)\Rightarrow (4)\right)$} Follows from the fact that $(T\mathcal{R}\otimes\C)\subset H_\bullet+(T\mathcal{R}\otimes\C)$.
     \medskip
     
    \textbf{$\left((4)\Rightarrow (2)\right)$} Follows from Proposition \ref{prop c}, using the fact that
    $$H_\bullet\subset H_\bullet+(T\mathcal{R}\otimes\C)\subset L_\bullet+(T\mathcal{R}\otimes\C)\,.$$
   \textbf{$\left((1)\Leftrightarrow (5)\right)$} By equation \eqref{eqc3}, condition $(1)$ implies condition $(5)$. Conversely, if condition $(5)$ holds, it implies condition $(4)$, as $T\mathcal{R}\otimes\C$ is Lie involutive. It results in condition $(1)$.
\end{proof}
\begin{remark}
There are some similar results in \cite[Propositions 4.9 and Theorem 4.1]{zambon1} related to conditions $(3)$ and $(4)$ in Theorem \ref{thm c}. However, the equivalent conditions $(1)$, $(2)$ and $(5)$ in Theorem \ref{thm c} are different and are more closely related to the intrinsic properties of generalized contact structures. Furthermore, using techniques of Proposition \ref{prop c} and Theorem \ref{thm c}, one may obtain analogous results to those of Theorem \ref{thm c} in the setting of real exact Courant algebroids and their isotropic subbundles, under suitable constant rank conditions on the intersection of bundles. 
\end{remark}
\begin{remark}
    Another sufficient condition for $H_\bullet$ in Theorem \ref{thm c} to be Courant involutive, is that $R$ preserves $C^{\infty}(L_\bullet)$ (cf. Definition \ref{imp def}), that is,
    \begin{equation}\label{inv1}
     [C^{\infty}(T\mathcal{R}\otimes\C),C^{\infty}(L_\bullet)]\subseteq C^{\infty}(L_\bullet)\,.
    \end{equation}
    This is immediate, as condition $(3)$ in Theorem \ref{thm c} holds after applying the equation \eqref{1.1}. However, converse may not be true, that is, when $R\not\in C^{\infty}(L_\bullet)$, it is still possible for $H_\bullet$ to be Courant involutive, without $L_\bullet$ being preserved by $R$. The following two examples demonstrate when this is always the case and when it is not. 
\end{remark}
\begin{example}
    Let $(M, \Phi, E_{\pm})$ be a generalized contact manifold of Poon-Wade type, with $E_-=R$, $E_+=\eta$, and $L_-$ the Courant involutive subbundle. Since $[C^{\infty}(T\mathcal{R}\otimes\C),C^{\infty}(L_-)]\subseteq C^{\infty}(L_-)$, the subbundle $H_-$ is Courant involutive.  Consider a closed real form $B\in\Omega^2(M)$ such that $B(R)\neq0$. Then $(e^B\Phi e^{-B},e^B E_{\pm})$ is a generalized contact structure with $e^BL_-$ as the corrsponding Courant involutive subbundle; see Example \ref{B-contact}. Using \cite[Proposition~3.23]{Gua}, we have $$[e^BC^{\infty}(H_-),e^BC^{\infty}(H_-)]\subseteq e^B[C^{\infty}(H_-),C^{\infty}(H_-)]$$ and so, $e^BH_-$ is also Courant involutive. Note that $e^B(T\mathcal{R}\otimes\C)=T\mathcal{R}\otimes\C$. It follows that
    \begin{align*}
        [C^{\infty}(T\mathcal{R}\otimes\C),e^BC^{\infty}(L_-)]&=[e^BC^{\infty}(T\mathcal{R}\otimes\C),e^BC^{\infty}(L_-)]\\
        &\subseteq e^B[C^{\infty}(T\mathcal{R}\otimes\C),C^{\infty}(L_-)]\\
        &\subseteq e^BC^{\infty}(L_-)
    \end{align*}
Hence, every generalized contact manifold of Poon-Wade type, as well as its $B$-field transformations, satisfies Condition \eqref{inv1}.
\end{example}
\begin{example}
Let $(\Phi,E_\pm)$ be the normal generalized contact structure on $\R^3$, as defined in Example \ref{new eq} where $E_+:=\frac{\partial}{\partial z}+dx$, and $E_-:=\frac{1}{2}(dz+\frac{\partial}{\partial x})$. Then the corresponding bundles $L_\pm$ (cf. \eqref{contact dirac}) are given by
$$L_\pm=\spn\left\{E_\pm\,,\frac{\partial}{\partial x}-dz-i\sqrt{2}dy\,,\frac{\partial}{\partial z}-dx+i\sqrt{2}\frac{\partial}{\partial y}\right\}\,.$$
In this case, $R=\frac{1}{2}\frac{\partial}{\partial x}+\frac{\partial}{\partial z}$. Choose a smooth function $f\in C^\infty(\R^3)$ such that $2\frac{\partial f}{\partial z}+\frac{\partial}{\partial x}\neq 0$. Then $$[R,f\left(\frac{\partial}{\partial x}-dz-i\sqrt{2}dy\right)]=-\left(\frac{\partial f}{\partial z}+\frac{1}{2}\frac{\partial}{\partial x}\right)(dz+i\sqrt{2}dy)\,.$$ 
Since $\langle E_+,dz+i\sqrt{2}dy\rangle\neq0$, we get $[C^{\infty}(T\mathcal{R}\otimes\C),C^{\infty}(L_+)]\not\subset C^{\infty}(L_+)$. Similarly, we can also get that $[C^{\infty}(T\mathcal{R}\otimes\C),C^{\infty}(L_-)]\not\subset C^{\infty}(L_-)$. Now, a straightforward computation shows that $H_\pm$ are of the form
\begin{align*}
 &H_+=\spn\left\{\sqrt{2}\frac{\partial}{\partial z}+i\frac{\partial}{\partial y}\,,\frac{\partial}{\partial x}+\frac{\partial}{\partial z}+dx-dz-i\sqrt{2}dy\right\}\,;\\
 &H_-=\spn\left\{\sqrt{2}\frac{\partial}{\partial x}-idy\,,\frac{\partial}{\partial x}+\frac{\partial}{\partial z}-dx+dz+i\sqrt{2}\frac{\partial}{\partial y}\right\}\,.
\end{align*}
It follows that both bundles $H_+,H_-$ are Courant involutive.
\end{example}
Although the following Lemma may be a well-known fact, we present a detailed argument here for the convenience of the reader.
\begin{lemma}\label{imp lemma1}
Let $f:N'\longrightarrow N$ be a surjective submersion (with connected fibers) between smooth manifolds. Let $E \to N'$ be a smooth vector bundle over $N'$. Then $E$ descends to a smooth vector bundle $F \to N$ (i.e., $E \cong f^* F$) 
if and only if there exists a flat partial connection on $E$ along the foliation $\mathcal{F'} := \ker(df)$, that is, $\mathcal{F'}$-connection, with trivial holonomy along the fibers of $f$.
\end{lemma}

\begin{proof}
Let $\mathcal{F'} := \ker(df) \subset TN'$ denote the foliation of $N'$ by fibers of $f$.  
\vspace{0.3em}

\textbf{($\Rightarrow$)} Suppose $E = f^* F$ for some smooth vector bundle $F \to N$, and let $\nabla^F$ be any connection on $F$. Define a partial connection (that is, $\mathcal{F'}$-connection) $\nabla^\mathcal{F'}$ on $E$ along $\mathcal{F'}$
\[
\nabla^\mathcal{F'}_X s := (f^* \nabla^F)_X s\,\,\,\text{for $X \in C^\infty(\mathcal{F'})$ and $s \in C^\infty(E)$\,.}
\] 
Note that $\nabla^\mathcal{F'}$ is the pullback connection $\nabla := f^* \nabla^F$ on $E$ restricted to $\mathcal{F'}$. Since $df|_{\mathcal{F'}}= 0$, the curvature $R^{\nabla^\mathcal{F'}}$ along $\mathcal{F'}$ vanishes, that is, $R^{\nabla^\mathcal{F'}}(X,Y)=0$ for $X,Y\in C^\infty(\mathcal{F'})$. In other words, it is flat. Moreover, parallel transport along any path $\gamma$ contained in a fiber is trivial, because the covariant derivative along any vertical vector vanishes. Thus, the holonomy along fibers is trivial. This gives a flat partial connection along fibers with trivial holonomy.
\medskip

\textbf{($\Leftarrow$)} Assume there exists a flat partial connection ($\mathcal{F'}$-connection) $\nabla^\mathcal{F'}$ on $E$ along $\mathcal{F'}$ with trivial holonomy. Let $y \in N$ and $x, x' \in f^{-1}(y)$ in the same fiber (leaf of $\mathcal{F'}$). For any smooth path $\gamma$ in $f^{-1}(y)$ from $x$ to $x'$, let
\[
P_\gamma: E_x \to E_{x'}
\]
denote the parallel transport along $\gamma$ via $\nabla^\mathcal{F'}$.  The flatness of $\nabla^\mathcal{F'}$ implies $P_\gamma$ depends only on the homotopy class of $\gamma$ inside the fiber. Furthermore, its trivial holonomy implies $P_\gamma$ is independent of the path. Hence, there is a canonical linear isomorphism
\begin{equation}\label{iso para}
    \phi_y:E_x \cong E_{x'}
\end{equation}
which depends smoothly on $x$ and $x'$. For $y \in N$, pick any $x \in f^{-1}(y)$ and define
\[
F_y := E_x\quad\text{and}\quad F:=\bigsqcup_{y\in N}F_y\,,
\]
using the canonical identification along the fiber. Then it is well-defined. Since the map $f:N'\longrightarrow N$ is a surjective submersion, it admits local smooth sections $s_U: U \subset N \to M$ of $f$. Define local trivializations of $F$ over $U$ by
$$F|_U := s_U^* E\,.$$
On overlaps $U \cap V$, the identification between different local sections is given by parallel transport along fibers, which is smooth and satisfies the cocycle condition due to flatness and trivial holonomy. This can be seen as follows.
\medskip 

Let $s,t: U \to M$ be smooth local sections of $f$. For each $y \in U$, consider the points $s(y), t(y) \in f^{-1}(y)$. It follows from \eqref{iso para} that there exists an unique canonical linear isomorphism 
$$\phi_y : E_{s(y)} \to E_{t(y)}$$
given by parallel transport along any path $\gamma$ in the fiber $f^{-1}(y)$ connecting $s(y)$ and $t(y)$ for all $y \in U$, and independent of path $\gamma$. Let $y \mapsto \gamma_y$ be a smooth family of paths in the fibers connecting $s(y)$ to $t(y)$. It induces a smooth family $y \mapsto \phi_y$, because in local product coordinates for $f$, the parallel transport maps are solutions of linear ODEs whose coefficients depend smoothly on the base point. This smooth dependence on parameters implies the resulting identifications vary smoothly (cf. \cite[Chapter~IV]{lang} or \cite[Chapter~9]{lee}). So, we obtain a canonical vector bundle isomorphism
\[
\hat\phi: s^* E \longrightarrow t^* E\quad\text{defined by\,\,$\hat{\phi}(y,v)=(y,\phi_y(v))$\,.}
\]
 For the cocycle condition on triple overlaps, let $s,t,r: U \to M$ be three local sections. Denote the canonical identifications along fibers by
\[
\hat\phi_{s,t}: s^* E \to t^* E, \quad\hat\phi_{t,r}: t^* E \to r^* E, \quad \hat\phi_{s,r}: s^* E \to r^* E\,.
\]
Take $v \in E_{s(y)}$. Since parallel transport along concatenated paths in a fiber is path-independent due to trivial holonomy, the composition of the parallel transports from $s(y)$ to $t(y)$ and from $t(y)$ to $r(y)$ yields a parallel transport from $s(y)$ to $r(y)$. Therefore,
$$\hat\phi_{t,r} \circ\hat\phi_{s,t} (v)=\hat\phi_{s,r}(v)\,.$$    
Therefore, the cocycle condition holds. 

\medskip
Hence, after gluing, we get that $F\to N$ admits a smooth vector bundle structure. For any $x \in M$ with $y = f(x)$, the fiber of the pullback bundle satisfies $(f^* F)_x = F_y = E_x$, under the canonical identification. Therefore, by construction of $F$, this defines a smooth vector bundle isomorphism between $E$ and $f^*F$. Hence $E$ descends to the vector bundle $F \to N$.
\end{proof}

\begin{remark}
The partial connection $\nabla^\mathcal{F'}$ is analogous to the \emph{Bott connection} (cf. \cite[Section 6]{bott71}) on the normal bundle of a foliation. Flatness along the leaves with trivial holonomy ensures that fibers of $E$ along each leaf are canonically identified, which is exactly the condition needed for descent.
\end{remark}
Consider the Bott connection, denoted by $\nabla$, on $\mathcal{N}_R$ (cf. \eqref{nrml bndl}) which is flat; see \cite[Lemma 6.2 \& 6.3]{bott71}. It follows that $\nabla$ induces a flat partial connection ($T\mathcal{R}$-connection) $\nabla^*$ on $\ann(R)$ by
$$\nabla_X^*\xi=\mathcal{L}_X\xi\,\,\,\text{for all $X\in C^\infty(T\mathcal{R})$, $\xi\in C^\infty(\ann(R))$}\,,$$
since $\mathcal{N}_R^*\cong\ann(R)\subset T^*M$. Consequently, we naturally obtain flat partial connections ($T\mathcal{R}$-connections), on $\mathbb{T}_{\mathcal{R}}M$ (cf. \eqref{subbndl}) and on its complexification $\mathbb{T}_{\mathcal{R}}M\otimes\C$, namely $\nabla\oplus\nabla^*$. For simplicity, both $T\mathcal{R}$-connections will be denoted by $\nabla^{\mathcal{R}}$. Suppose that the leaf space $M/\mathcal{R}$ is a smooth manifold, that is, $(M^{2n+1},\Phi,E_{\pm})$ is regular; see Definition \ref{def:reg}. Then the holonomy of $\nabla$ is trivial. Consequently, the induced partial connection $\nabla^{\mathcal{R}}$ on $\mathbb{T}_{\mathcal{R}}M$ has trivial holonomy. Therefore, by Lemma \ref{imp lemma1}, $\mathbb{T}_{\mathcal{R}}M$ descends to a smooth
vector bundle $E_{red}\to M/\mathcal{R}$, that is, $\mathbb{T}_{\mathcal{R}}M\cong\Psi^*E_{red}$. Here $\Psi:M\longrightarrow M/\mathcal{R}$ is the canonical projection as defined in \eqref{psi}. Now, it follows from the proof of \textit{Theorem 3.7} in \cite{zambon} that the bundle $E_{red}$ is, in fact, an untwisted exact Courant algebroid (cf. \cite{burst,zambon})  over $M/\mathcal{R}$, reduced from $\mathbb{T}_{\mathcal{R}}M$, namely, $E_{red}=T(M/\mathcal{R})\oplus T^*(M/\mathcal{R})$; see also \cite[Example 3.11]{burst} \cite[Example 3.9]{zambon}). In particular, we also get $$\mathbb{T}_{\mathcal{R}}M\cong\Psi^*(T(M/\mathcal{R}))\oplus\Psi^*(T^*(M/\mathcal{R})),\,\,\,\text{and}\,\,\,C^{\infty}(E_{red})\cong C^{\infty}_{\mathcal{R}}(T\mathcal{R}^{\perp})/C^{\infty}(T\mathcal{R})\,.$$
Further assume that $[C^\infty(T\mathcal{R}\otimes\C),C^\infty(H_\bullet)]\subseteq C^\infty(H_\bullet)$ where $H_\bullet$ as in Theorem \ref{thm c}, and let $X+\xi,Y+\beta\in C^\infty(H_\bullet)$. Consider the identification in \eqref{identi}. Then, applying Remark \ref{imp rmk1}, $H_\bullet$ can be identified as a subbundle of $\mathbb{T}_{\mathcal{R}}M\otimes\C$, when $R\not\in C^{\infty}(L_\bullet)$, namely $H_\bullet\cong (H_\bullet+(T\mathcal{R}\otimes\C))/T\mathcal{R}\otimes\C$. Then  
\begin{equation}\label{eq1.3}
\begin{aligned}
\langle\nabla^{\mathcal{R}}_R(X+\xi),Y+\beta\rangle&=\langle\nabla_RX+\nabla^{*}_R\xi,Y+\beta\rangle\\
&=\frac{1}{2}(\beta(\nabla_RX)+\mathcal{L}_R\xi(Y))\\
&=\frac{1}{2}(\beta([R,X])+\mathcal{L}_R\xi(Y))\quad(\text{as\, $\beta\in C^\infty(\ann(R)$}))\\
&=\langle [R,X]+\mathcal{L}_R\xi,Y+\beta\rangle\\
&=\langle [R, X+\xi],Y+\beta\rangle\\
&=0\,.\quad(\text{as\, $H_\bullet$ is isotropic)}
\end{aligned}
\end{equation}
Therefore, $\nabla^{\mathcal{R}}_R(X+\xi)\in C^\infty(H^\perp_\bullet)$, and by Proposition \ref{prop c}, it implies that $\nabla^{\mathcal{R}}_R(X+\xi)\in C^\infty(H_\bullet)$. Hence
\begin{equation}\label{eq1.4}
\nabla^{\mathcal{R}}_ss'\in C^\infty((H_\bullet+(T\mathcal{R}\otimes\C))/T\mathcal{R}\otimes\C)\,,
\end{equation}
for all $s'\in C^\infty((H_\bullet+(T\mathcal{R}\otimes\C))/T\mathcal{R}\otimes\C)$, $s\in C^\infty(T\mathcal{R}\otimes\C)$. 
\medskip

Now suppose $R\in C^{\infty}(L_\bullet)$. In this case, we observe that $H_\bullet=L_\bullet$ (cf.\eqref{eq1.5}). This implies that the eigenbundle $E^{(*,*)}$ (cf. \eqref{contact dirac}) of  $\Phi$ associated with $L_\bullet$, can be identified as a subbundle of $\mathbb{T}_{\mathcal{R}}M\otimes\C$, as $E^{(*,*)}\cong (L_\bullet+(T\mathcal{R}\otimes\C))/T\mathcal{R}\otimes\C$; cf. Remark \ref{imp rmk1}. Here $(*,*)\in\{(1,0),(0,1)\}$. By assumption and Theorem \ref{thm c}, we get $[C^\infty(T\mathcal{R}\otimes\C),C^\infty(E^{(*,*)})]\subseteq C^\infty(E^{(*,*)})$ when $E^{(*,*)}\subset\mathbb{T}_{\mathcal{R}}M\otimes\C$ is regarded as a subbundle. Proceeding in a similar manner, we obtain equation \eqref{eq1.4}, since $E^{(*,*)}$ is isotropic.
\medskip

Hence the bundle $(H_\bullet+(T\mathcal{R}\otimes\C))/T\mathcal{R}\otimes\C$ admits the flat $T\mathcal{R}$-connection $\nabla^\mathcal{R}$ with trivial holonomy. By Lemma \ref{imp lemma1}, this bundle descends to a smooth complex subbundle $L_{red}\subset E_{red}\otimes\C$ over $M/\mathcal{R}$, so that $(H_\bullet+(T\mathcal{R}\otimes\C))/T\mathcal{R}\otimes\C\cong\Psi^*L_{red}$. In fact, the subbundle $L_{red}$ is Courant involutive and maximal isotropic, as it is reduced from $H_\bullet+(T\mathcal{R}\otimes\C)$ which is Courant involutive by Theorem \ref{thm c}, and it also satisfies $L_{red}\cap\overline{L_{red}}=\{0\}$ by Remark \ref{imp rmk1}. Therefore, $L_{red}$ defines a generalized complex structure (cf. Definition \ref{gcs}), denoted by $\mathcal{J}_{red}$ on $M/\mathcal{R}$, whose $+i$-eigenbundle is precisely $L_{red}$. Thus, we have proved the following.
\begin{theorem}\label{thm c1}
Let $(M^{2n+1}, \Phi, E_{\pm})$ be a generalized contact manifold with Courant involutive bundle $L_\bullet$ (cf. \eqref{contact dirac}), where $\bullet\in\{+,-\}$. Let $\{R, \eta\}$ be the corresponding nowhere vanishing pair (see \eqref{f-eta}), with the foliation $\mathcal{R}$ induced by $R$. Let $H_\bullet$ be the bundle associated with $L_\bullet$, as defined in Proposition \ref{prop c}. Let $\mathbb{T}_{\mathcal{R}}M$ be defined as in \eqref{subbndl}. Assume that $H_\bullet$ is Courant involutive. Then,
\vspace{0.2em}
\begin{enumerate}
\setlength\itemsep{0.5em}
    \item Both vector bundles  $\mathbb{T}_{\mathcal{R}}M$ and $(H_\bullet+(T\mathcal{R}\otimes\C))/T\mathcal{R}\otimes\C$ admit partial connections along $\mathcal{R}$, that is, $T\mathcal{R}$-connections, which are flat and have trivial holonomy. In fact, the $T\mathcal{R}$-connection on $\mathbb{T}_{\mathcal{R}}M\otimes\C$ restricts naturally to $(H_\bullet+(T\mathcal{R}\otimes\C))/T\mathcal{R}\otimes\C$.
    \item The normal bundle $\mathcal{N}_R$ (cf. \eqref{nrml bndl}) admits a generalized complex structure. More precisely, $(H_\bullet+(T\mathcal{R}\otimes\C))/T\mathcal{R}\otimes\C$ induces a generalized complex structure on $\mathbb{T}_{\mathcal{R}}M\otimes\C$.
\end{enumerate}
\vspace{0.2em} 
Furthermore, if $(M^{2n+1}, \Phi, E_{\pm})$ is also regular so that $\mathcal{R}$ is simple, the following holds.
\vspace{0.2em}
\begin{enumerate}
\setlength\itemsep{0.5em}
    \item[(3)] $\mathbb{T}_{\mathcal{R}}M$ descends to a smooth vector bundle $E_{red}\to M/\mathcal{R}$, that is, $\mathbb{T}_{\mathcal{R}}M\cong\Psi^*E_{red}$. Moreover, $E_{red}\cong T(M/\mathcal{R})\oplus T^*(M/\mathcal{R})$ with the symmetric bilinear form and the Courant bracket defined as in \eqref{bilinear} and \eqref{bracket}, respectively.
    \item[(4)]  $(H_\bullet+(T\mathcal{R}\otimes\C))/T\mathcal{R}\otimes\C$ descends to a maximal isotropic complex subbundle $L_{red}\to M/\mathcal{R}$ of $E_{red}\otimes\C$ which is Courant involutive, and satisfies $L_{red}\cap\overline{L_{red}}=\{0\}$. In other words, $(H_\bullet+(T\mathcal{R}\otimes\C))/T\mathcal{R}\otimes\C\cong\Psi^*L_{red}$.
    \item[(5)] $M/\mathcal{R}$ admits a generalized complex structure $\mathcal{J}_{red}$.
\end{enumerate}
Here, $M/\mathcal{R}$ denotes the leaf space of the foliation $\mathcal{R}$, and $\Psi:M\longrightarrow M/\mathcal{R}$ is the canonical projection, defined as in \eqref{psi}.
\end{theorem}
Continue with the regular generalized contact manifold $(M^{2n+1}, \Phi, E_{\pm})$, as in Theorem \ref{thm c1}. Let  
\begin{equation}\label{pr}
\begin{aligned}
 &\pr_{TM}:(TM\oplus T^*M)\otimes\C\longrightarrow TM\otimes\C\,;\\
 &\pr_{T(M/\mathcal{R})}:(T(M/\mathcal{R})\oplus T^*(M/\mathcal{R}))\otimes\C\rightarrow T(M/\mathcal{R})\otimes\C\,.
\end{aligned}
\end{equation}
be the projection maps onto $TM\otimes\C$ and $T(M/\mathcal{R})\otimes\C$, respectively. Consider the identification in \eqref{identi} and set
\begin{equation}\label{k}
\begin{aligned}
 &K_\bullet:=\pr_{TM}\left((H_\bullet+(T\mathcal{R}\otimes\C))/T\mathcal{R}\otimes\C\right)\subset\mathcal{N}_R\otimes\C\,;\\
 &K_{red}:=\pr_{T(M/\mathcal{R})}(L_{red})\subset T(M/\mathcal{R})\otimes\C\,,
 \end{aligned}
\end{equation}
where $\bullet\in\{+,-\}$. Then there exist two real distributions $\Delta_\bullet\subset\mathcal{N}_R\subset TM$ and $\Delta_{red}\subset T(M/\mathcal{R})$ such that
\begin{equation}\label{k1}
\begin{aligned}
&\Delta_\bullet\otimes\C=K_\bullet\cap\overline{K_\bullet}\,;\\
&\Delta_{red}\otimes\C=K_{red}\cap\overline{K_{red}}\,.
\end{aligned}
\end{equation}
Since $K_\bullet$ is Courant involutive, it follows that $\Delta_\bullet$ is also Courant involutive. Moreover, we have
$$K_\bullet\cong\Psi^*K_{red}\quad\text{and}\quad\Delta_\bullet\cong\Psi^*\Delta_{red}\quad(\text{as distributions})\,.$$ 
Now the distribution $\Delta_{red}$ defines a Poisson structure on $M/\mathcal{R}$, determined by $\mathcal{J}_{red}$; see \cite[Section~3.4]{Gua2}, \cite[Proposition 1.13]{bailey}. In particular, it induces a symplectic foliation (possibly singular), denoted by $\mathcal{S}$, and the codimension of the symplectic leaf through a point $y\in M/\mathcal{R}$ is equal to $(2\type(y))$ (cf. Definition \ref{def:type}). Denote by $T\mathcal{S}$ the tangent distribution of the foliation $\mathcal{S}$, and by $N\mathcal{S}\subset T(M/\mathcal{R})$ a corresponding normal distribution (that is, a complementary $T\mathcal{S}$). Equivalently, one may write
\begin{equation}\label{distri}
    T(M/\mathcal{R})=T\mathcal{S}\oplus N\mathcal{S}\,.
\end{equation}
Moreover $N_y\mathcal{S}$ admits a linear complex structure for $y\in M/\mathcal{R}$; see \cite[Proposition 1.13]{bailey}. Define a function $\rho_\bullet:M\to[0,\infty)$ by
\begin{equation}\label{rho1}
\rho_\bullet(x):=\frac{1}{2}\dim_{\R}(\Delta_\bullet)_x\cong\frac{1}{2}\dim_{\R}(\Delta_{red})_{\Psi(x)}=n-\type(\Psi(x))\,,    
\end{equation}
where $x\in M$. Suppose $M\xrightarrow{\Psi} M/\mathcal{R}$ is a principal bundle where $\eta$ as its connection $1$-form, given as in Theorem \ref{thm1} with $[R,\eta]=0$. It follows that $\ker(\eta)\subset TM$, which is the Horizontal bundle for the connection, is exactly the normal bundle $\mathcal{N}_R$. Because $d\eta\in\en(TM,T^*M)$ and $\mathcal{N}^*_R=\ann(R)$, we get $d\eta\in\en(\mathcal{N}_R,\mathcal{N}^*_R)$. Here, $\en(V,V')$ denotes the set of endomorphisms between two vector bundles $V,V'$. Let $\rho_\bullet>0$ on $M$, and suppose further that,
\begin{equation}\label{conditn}
\eta\wedge(d\eta)^{\rho_\bullet(x)}|_{(\Delta_\bullet)_x}\neq 0\quad\text{at each $x\in M$}\,.    
\end{equation}
Consequently, $(d\eta)_x$ is nondegenerate on $(\Delta_\bullet)_x$ for all $x\in M$. Since $\Delta_\bullet\cong\Psi^*\Delta_{red}$ and $\Psi$ is a surjective submersion, the curvature $\omega\in\Omega^2(M/\mathcal{R})$ is therefore nondegenerate on $(\Delta_{red})_y$ for all $y\in M/\mathcal{R}$. So, $\omega$ yields a the  symplectic foliation (possibly singular), denoted by $\mathcal{S}'$, on $M/\mathcal{R}$. Also, since $\omega$ is closed, $\ker(\omega)\subset T(M/\mathcal{R})$ provides a complementary foliation (possibly singular) of $\mathcal{S}'$ which is (pointwise) isomorphic to $N\mathcal{S}\subset T(M/\mathcal{R})$. In other words, 
\begin{equation}\label{distri1}
    T(M/\mathcal{R})=T\mathcal{S}'\oplus \ker(\omega)\,,\,\,\,\text{and}\,\,\,\ker(\omega)\cong N\mathcal{S}\,\,(\text{pointwise})\,,
\end{equation} where $T\mathcal{S}'$ is the tangent distribution of the foliation $\mathcal{S}$.
\medskip

When $\rho_\bullet \equiv 0$, the reduced structure $\mathcal{J}_{\mathrm{red}}$ can coincide, up to a $B$-field transformation, with a generalized complex structure induced by an ordinary complex structure on $M/\mathcal{R}$ as in Example \ref{complx eg}; see \cite[Example 4.21]{Gua}. Particularly, $M/\mathcal{R}$ is a complex manifold. In this case, $K_\bullet\subset\mathcal{N}_R\otimes\C$ is a Lie involutive complex subbundle satisfying $$K_\bullet\oplus\overline{K_\bullet}=\mathcal{N}_R\otimes\C\,,$$ which shows that the normal bundle $\mathcal{N}_R$ admits a complex structure. Since $\ker(\eta)=\mathcal{N}_R$, it follows that $d\eta(X,Y)=0$ for all $X,Y\in C^\infty(K_\bullet)$. Because $d\eta$ is a real $2$-form, and $K_\bullet^*=\overline{K_\bullet}$, we get that  $$d\eta\in C^{\infty}(K_\bullet\otimes\overline{K_\bullet})\,,$$ that is, $d\eta$ has no $(2,0)$ or $(0,2)$ component. Consequently, the curvature form $\omega$ is of type $(1,1)$ with respect to the induced complex structure.  

\medskip
When $\rho_\bullet>0$ (cf. \eqref{rho1}) is constant, that is, $\mathcal{J}_{red}$ is a regular generalized complex (GC) structure (cf. Definition \ref{def:type}), the associated symplectic foliations $\mathcal{S}$, $\mathcal{S}'$ are regular. The normal distribution of $\mathcal{S}$ becomes a subbundle, and admits an integrable complex structure; see \cite[Proposition]{Gua2}. In this case, the pair $(\mathcal{S}',\ker(\omega))$ defines a GC structure by \cite[Corollary~2.7]{bailey}, since $\omega$ vanishes along $\ker\omega$. Let us denote this structure by $\mathcal{J}_{M/\mathcal{R}}$. By \cite[Proposition~2.5]{bailey}, there exists a real form $B\in\Omega^2(M/\mathcal{R})$ such that
\begin{equation}\label{b1}
e^{B}\mathcal{J}_{red}\,e^{-B} = \mathcal{J}_{M/\mathcal{R}}\,,    
\end{equation}
where $e^B$ as defined in \eqref{B transformation}. Since both $\mathcal{J}_{red}$ and $\mathcal{J}_{M/\mathcal{R}}$ are GC structures, and only real closed $2$-forms preserve the Courant bracket, it follows that $B$ must be closed. Hence $\mathcal{J}_{red}$ and $\mathcal{J}_{M/\mathcal{R}}$ coincide up to a $B$-field transformation.

\begin{theorem}(Generalized Boothby-Wang Theorem)\label{thmc2}
Let $(M, \Phi, E_{\pm})$ be a regular generalized contact manifold, with notation and assumptions as in Theorem \ref{thm c1}. Suppose that the vector field $R$ is also complete and satisfies $[R,\eta]=0$. Then $M/\mathcal{R}$ admits a generalized complex (GC) structure $\mathcal{J}_{red}$, and the projection $M\xrightarrow{\Psi} M/\mathcal{R}$ defines a principal $\R$-or $S^1$-bundle structure on $M$ over the GC manifold $(M/\mathcal{R},\mathcal{J}_{red})$, with $\eta$ as a connection $1$-form and curvature $\omega\in\Omega^2(M/\mathcal{R})$ given by $d\eta=\Psi^*\omega$. 
\medskip

Moreover, for $\rho_\bullet>0$ on $M$, at each point $x\in M$, if $$\eta\wedge(d\eta)^{\rho_\bullet(x)}|_{(\Delta_\bullet)_x}\neq 0\,,$$ where $\rho_\bullet(x)=\dim_{\R}(\Delta_\bullet)_x$ and $\Delta_\bullet$ as in \eqref{k1}, then
\vspace{0.2em}
\begin{enumerate}
\setlength\itemsep{0.5em}
    \item The curvature $\omega$ determines a symplectic foliation (possibly singular) $\mathcal{S}'$, whose complementary foliation is given by $\ker(\omega)$.
    \item When $\mathcal{J}_{red}$ is a regular GC structure (that is, $\rho_\bullet$ is constant), the pair $(\mathcal{S}', \ker(\omega))$ defines another GC structure $\mathcal{J}_{M/\mathcal{R}}$ such that $$e^{B}\mathcal{J}_{red}\,e^{-B} = \mathcal{J}_{M/\mathcal{R}}\,,$$ for some real closed form $B\in\Omega^2(M/\mathcal{R})$.
\end{enumerate}
If $\rho_\bullet=0$ on $M$, then the leaf space $M/\mathcal{R}$ admits a complex structure $J_{M/\mathcal{R}}$, and  the curvature $\omega$ is of type $(1,1)$ with respect to $J_{M/\mathcal{R}}$. Additionally, the GC structure $\mathcal{J}_{\mathrm{red}}$ may be given by a $B$-field transformation of the GC structure defined by $J_{M/\mathcal{R}}$ as in Example \ref{complx eg}.
\end{theorem}
\begin{proof}
    Follows from the preceding discussion, and by applying Theorem \ref{thm2} and Theorem \ref{thm c1}. 
\end{proof}
The following corollary provides a direct generalization of \cite[Theorem 2]{bw}.
\begin{cor}
    Notations and assumptions are as in Theorem \ref{thmc2}, except without assuming $R$ is complete. Given any compact regular generalized contact manifold $(M, \Phi, E_{\pm})$, $M\rightarrow M/\mathcal{R}$ is a principal $S^1$-bundle over the generalized complex (GC) manifold $(M/\mathcal{R},\mathcal{J}_{red})$, where $\eta$ is a connection $1$-form whose curvature can induce the symplectic foliation associated with $\mathcal{J}_{red}$.  
\end{cor}
\begin{proof}
    Follows from Theorem \ref{thmc2}, since every nowhere vanishing vector field on $M$ is complete.
\end{proof}
\begin{remark}
Notice that $T_y\mathcal{S}' = T_y\mathcal{S}$ (cf. \eqref{distri}, \eqref{distri1}) for all $y \in M/\mathcal{R}$. It follows that $\ker(\omega)_y$ carries a linear complex structure, since $\ker(\omega)_y \cong N_y\mathcal{S}$. Consequently, the pair $(\omega_y, \ker(\omega)_y)$ defines a linear generalized complex structure (cf. \cite[Definition~4.1]{Gua}) on $T_y(M/\mathcal{R})$, which we denote by $(\mathcal{J}_{M/\mathcal{R}})_y$; see also \cite[Theorem~4.13]{Gua}. Then, applying \cite[Proposition~2.5]{bailey} pointwise, one obtains at each $y$ a real 2-form $B_y \in \wedge^2 T_y^*(M/\mathcal{R})$ such that
$$e^{B_y} (\mathcal{J}_{red})_y\, e^{-B_y} = (\mathcal{J}_{M/\mathcal{R}})_y\,.$$
However, both families $\{B_y\}_{y\in M/\mathcal{R}}$ and $\{(\mathcal{J}_{M/\mathcal{R}})_y\}_{y\in M/\mathcal{R}}$ may fail to vary smoothly whenever $(M,\mathcal{J}_{red})$ is not a regular GC manifold.
\end{remark}
\begin{remark}\label{rmk3}
    The key idea behind condition \ref{conditn} is to ensure that $(d\eta)_x|_{(\Delta_\bullet)_x}$ is nondegenerate, which in turn implies that the curvature induces the symplectic foliation. In other words, Theorem \ref{thmc2} remains valid if we assume only that $(d\eta)_x|_{(\Delta_\bullet)_x}$ is nondegenerate whenever $\rho_\bullet>0$ on $M$, without requiring condition \ref{conditn} in its entirety.
\end{remark}

\subsection{Converse of Theorem \ref{thmc2}}\label{convers}

Following \cite[Chapter~6]{tu17}, we recall some basic notions regrading connections on a principal bundle that will be needed later for the converse part. Let $P\xrightarrow{\pi} N$ be a principal $S^1$-bundle over a smooth manifold $N$. Let $\eta$ be a connection $1$-form on $P$. Let $R\in TP$ denote the (unique) fundamental vector field, normalized by $\eta(R)=1$. Here, we identify the Lie algebra of $S^1$ with $\R$. The connection $\eta$ defines the horizontal and vertical subbundles $H$ and $V$ of $TP$, respectively, yielding the decomposition
\begin{equation}\label{conn eq1}
    TP=H\oplus V,\quad\text{with}\,\,\,H=\ker(\eta),\,\, H^*=\ann(R),\,\,\text{and}\,\,V=\ker(\pi)=\spn\{R\}\,.
\end{equation}
For any $X\in TN$, denote by $\widetilde{X}\in H\subset TP$ the unique horizontal lift of $X$. Then, the horizontal lifting induces a bundle isomorphism $s:\pi^*(TN)\rightarrow H\subset TP$ which further induces a bundle isomorphism
\begin{equation}\label{conn iso}
    \tilde{s}:=s\oplus (s^*)^{-1}:\pi^*(TN)\oplus\pi^*(T^*N)\rightarrow H\oplus H^*\subset TP\oplus T^*P\,,
\end{equation}
given by $\tilde{s}(X+\xi)=s(X)+(s^*)^{-1}(\xi)$ where $s^*: H^*\rightarrow\pi^*(T^*N)$ is the dual isomorphism of $s$, defined similarly as in Remark \ref{rmk notatn}. For simplicity, we also denote the complexification of $\tilde{s}$ by $\tilde{s}$.
\medskip

Now, let $N^{2n}$ be a GC manifold with the maximal, isotropic and Courant involutive subbundle $L_N$ as in Definition \ref{gcs}. Define 
\begin{equation}\label{conn eq2}
    \begin{aligned}
        &\widetilde{L}:=\left\{\widetilde{X}+\pi^*\xi\in(H\oplus H^*)\otimes\C\mid X+\xi\in L_N\,\,\text{and $\widetilde{X}$ is the unique horizontal lift of $X$}\right\}\,,\\
        &L^+:=\left(\spn\{R\}\otimes\C\right)\oplus\widetilde{L}\,,\,\,\,\text{and}\,\,\,L^-:=\left(\spn\{\eta\}\otimes\C\right)\oplus\widetilde{L}\,.
    \end{aligned}
\end{equation}
\begin{lemma}\label{conn lem1}
~
    \begin{enumerate}
    \setlength\itemsep{0.4em} 
        \item $\widetilde{L}\subset(TP\oplus T^*P)\otimes\C$ is an isotropic complex subbundle of rank $2n$.
        \item $L^\pm$ are maximal isotropic complex subbundles of rank $2n+1$. Particularly, $L^+$ is closed under the Courant bracket \eqref{bracket}.
        \item Whenever $d\eta=0$, the subbundles $\widetilde{L}$ and $L^-$ are Courant involutive.
    \end{enumerate}
\end{lemma}
\begin{proof}
    $(1)$ Since the restriction $\tilde{s}|_{\pi^*L_N}:\pi^*L_N\rightarrow\tilde{s}(\pi^*L_N)\subset(TP\oplus T^*P)\otimes\C$ of $\tilde{s}$ in \eqref{conn iso} is an bundle isomorphism, and $\tilde{s}(\pi^*L_N)=\widetilde{L}$, it follows that $\widetilde{L}$ is a complex subbundle of rank $2n$. A direct computation shows that $\widetilde{L}$ is also isotropic with respect to the bilinear form \eqref{bilinear}.
    
    \medskip
    \hspace{-1em}$(2)$ Using $(1)$, a straightforward computation shows that $L_\pm$ are maximal isotropic complex subbundles of $(TP\oplus T^*P)\otimes\C$. For Courant involutiveness, let $\widetilde{X}+\pi^*\xi, \widetilde{Y}+\pi^*\beta\in C^\infty(\widetilde{L})$ with corresponding $X+\xi,Y+\beta\in C^\infty(L_N)$, respectively. Note that, by \eqref{conn eq1}, $d\eta(\widetilde{X},\widetilde{Y})=-\eta(\mc{L}_{\widetilde{X}}\widetilde{Y})$. Then 
    \begin{equation}\label{conn imp eq}
    \begin{aligned}
[\widetilde{X}+\pi^*\xi,\widetilde{Y}+\pi^*\beta]&=\mc{L}_{\widetilde{X}}\widetilde{Y}+\mc{L}_{\widetilde{X}}\pi^*\beta-\mc{L}_{\widetilde{Y}}\pi^*\xi-\frac{1}{2}d(i_{\widetilde{X}}\pi^*\beta-i_{\widetilde{Y}}\pi^*\xi)\\
&=\widetilde{\mc{L}_XY}-d\eta(\widetilde{X},\widetilde{Y})R+\mc{L}_{\widetilde{X}}\pi^*\beta-\mc{L}_{\widetilde{Y}}\pi^*\xi-\frac{1}{2}d(i_{\widetilde{X}}\pi^*\beta-i_{\widetilde{Y}}\pi^*\xi)\\
&=\widetilde{\mc{L}_XY}-d\eta(\widetilde{X},\widetilde{Y})R+\pi^*\left(\mc{L}_{X}\beta-\mc{L}_{Y}\xi-\frac{1}{2}d(i_{X}\beta-i_{Y}\xi)\right)\,.\\
[\widetilde{X}+\pi^*\xi,R]&=-\mc{L}_R\widetilde{X}\quad(\text{as $R\in\ker(\pi)$})\,.\\
[\widetilde{X}+\pi^*\xi,\eta]&=\mc{L}_{\widetilde{X}}\eta=i_{\widetilde{X}}d\eta\quad(\text{as $H=\ker(\eta)$})\,.
    \end{aligned}
    \end{equation}
From \eqref{conn imp eq}, it follows that $[\widetilde{X}+\pi^*\xi,\widetilde{Y}+\pi^*\beta]\in C^\infty(L^+)$, implying 
\begin{equation}\label{cou}
 [C^\infty(\widetilde{L}),C^\infty(\widetilde{L})]\subseteq C^\infty(L^+)\,.   
\end{equation}
Consequently,
\begin{equation*}
    \begin{aligned} &\left\langle[\widetilde{X}+\pi^*\xi,\widetilde{Y}+\pi^*\beta],R\right\rangle=0,\\
    \implies&\left\langle\widetilde{Y}+\pi^*\beta,[\widetilde{X}+\pi^*\xi,R]\right\rangle=\widetilde{X}\left\langle\widetilde{Y}+\pi^*\beta,R\right\rangle,\quad(\text{by \eqref{eqcc}})\\
    \implies&\left\langle\widetilde{Y}+\pi^*\beta,[\widetilde{X}+\pi^*\xi,R]\right\rangle=0\,.
    \end{aligned}
\end{equation*}
Since $\widetilde{Y}+\pi^*\beta\in C^\infty(\widetilde{L})$ is arbitrary, it implies that $[\widetilde{X}+\pi^*\xi,R]\in C^\infty(L^+)$. A similar computation shows that $[\widetilde{X}+\pi^*\xi,\eta]\in C^\infty(L^-)$. Therefore,
\begin{equation}\label{conn imp eq1}
    [C^\infty(\widetilde{L}),R]\in C^\infty(L^+),\,\,\,\text{and}\,\,\,[C^\infty(\widetilde{L}),\eta]\in C^\infty(L^-)\,.
\end{equation}
Hence, combining \eqref{cou} and \eqref{conn imp eq1}, we obtain that $L^+$ is Courant involutive.

\medskip
    \hspace{-1em}$(3)$ Follows immediately from \eqref{conn imp eq} and \eqref{conn imp eq1} in $(2)$.
\end{proof}
\begin{theorem}(Converse of Theorem \ref{thmc2})\label{conn thm}
    Let $P\xrightarrow{\pi} N$ be a principal $S^1$-bundle over a $2n$-dimensional GC manifold $N$. Fix any connection $1$-form $\eta$ on $P$, with the unique fundamental vector field $R$ satisfying $\eta(R)=1$. Let $\Omega\in\Omega^2(N)$ be the curvature $2$-form, given by $d\eta=\pi^*\Omega$. Consider the subbundles $\widetilde{L},L^\pm\subset(TP\oplus T^*P)\otimes\C$, defined as in \eqref{conn eq2}.  Then $P$ admits a canonical regular generalized contact structure of Poon-Wade type $(\Phi,E_\pm)$ where $E_+=R$ and $E_-=\eta$ (cf. Definition \ref{def:contct}). Particularly, this structure is determined by the quintuple $(\eta,R,\widetilde{L},L^\pm)$. Moreover, when $\Omega=0$, the triple $(\Phi,E_\pm)$ defines a normal generalized contact structure of Poon-Wade type.
\end{theorem}
\begin{proof}
    Set $E_+:=R$, $E_-:=\eta$, and consider the bundle endomorphism on $(TP\oplus T^*P)\otimes\C$, defined by
    $$\hat{\Phi}:=\begin{cases}
           0 & \text{on $\spn\{E_\pm\}\otimes\C$}\,;\\
           i I_{\widetilde{L}} & \text{on $\widetilde{L}$}\,;\\
           -i I_{\overline{\widetilde{L}}} & \text{on $\overline{\widetilde{L}}$}\,\,,
        \end{cases}$$
where $\overline{\widetilde{L}}$ is the conjugate bundle of $\widetilde{L}$, and $I_{\bullet}\in\en(\bullet)$ denotes the identity automorphism. Notice that, the triple $(\hat{\Phi},E_\pm)$ satisfies all the conditions of Definition \eqref{cntct def}, \textit{albeit} after complexification. Consequently, the restriction of $\hat{\Phi}$ to the underlying real subbundle $TP\oplus T^*P$ induces the required endomorphism. Since, $\hat{\Phi}$ is real, in the sense that $\hat{\Phi}(\bar{e})=\overline{\hat{\Phi}(e)}$ for all $e\in(TP\oplus T^*P)\otimes\C$, it preserves $TP\oplus T^*P$. Therefore, the restriction
$$\Phi:=\hat{\Phi}|_{TP\oplus T^*P}:TP\oplus T^*P\rightarrow TP\oplus T^*P$$ is well-defined, and so, the triple $(\Phi,E_\pm)$ defines the almost generalized contact structure of Poon-Wade type on $P$. Now, the $\pm i$-eigenbundles of $\hat{\Phi}$ are $\widetilde{L}, \overline{\widetilde{L}}$, by construction, respectively. It implies that $L^{\pm}=L_{\pm}$ where $L_{\pm}$ are defined as in \eqref{contact dirac}. By Lemma \ref{conn lem1}, $L_+$ is Courant involutive, implying that $(\Phi,E_\pm)$ is a regular generalized contact structure of Poon-Wade type. Since $[E_+,E_-]=0$, it follows from $(3)$ of Lemma \ref{conn lem1} that the generalized contact structure is normal (see Definition \ref{def:contct}) whenever the curvature, induced by the connection $\eta$, vanishes.
\end{proof}
\begin{remark}\label{spcl rmk}
In Theorem \ref{conn thm}, when the GCS is induced by a (integral) symplectic form (cf. Example \ref{symplectic eg}), we recover the converse \cite[Theorem~3]{bw} of the classical Boothby-Wang theorem (\cite[Theorem~2]{bw}). Similarly, when the GCS is given by a complex structure (cf. Example \ref{complx eg}) and the curvature form is of type $(1,1)$, the generalized contact structure in Theorem \ref{conn thm} is determined by a normal almost contact structure as in Example \ref{almost eg}, thereby recovering the classical result of Morimoto \cite[Theorem~6]{mori1}. 
\end{remark}
\begin{remark}
A necessary and sufficient condition for existence of a nontrivial, non-flat principal $S^1$-bundle over a GC manifold $N$ is $H^2(N,\Z)\neq 0$; see \cite[Sections~2-3]{koba}. Thus, Theorem \ref{conn thm} implies that for every nonzero integral class in $H^2(N,\Z)$, we have a principal $S^1$-bundle with a canonical generalized contact structure of Poon-Wade type. So, It is interesting to classify these structures in terms of the integral cohomology classes represented by their curvature forms which we intend to address in future works.
\end{remark}
\section{Observations and examples}
In this section, we focus on two things. First, we consider certain generalized contact manifolds and show how both Theorem \ref{thm c1} and Theorem \ref{thmc2} apply to them. This demonstrates that Theorem \ref{thmc2} generalizes the classical Boothby-Wang construction and extends related results in contact and almost contact geometry to generalized geometry. Second, we give a construction for new generalized contact manifolds that are not of Poon-Wade type. Notations are as in Section \ref{sec3}.
\subsection{Generalized contact manifolds of Poon-Wade type}
Let $(M, \Phi, E_{\pm})$ be a generalized contact manifold of Poon-Wade type; see Definition \ref{def:p-w}. Without loss of generality, assume that $E_+\in C^\infty(T^*M)$, $E_-\in C^\infty(TM)$, and that $L_-$ is Courant involutive. Then $R=E_-$ and $\eta=E_+$; cf. \eqref{f-eta}. Moreover, $R\in C^\infty(L_-)$, since $$L_-=(L_{E_-}\otimes\C)\oplus E^{(1,0)},\,\,\,L_{E_-}=T\mathcal{R}\,\,\,\text{and}\,\,\,E^{(1,0)}=\big\{e-i\Phi(e)\,|\,e\in\ker(\eta)\oplus\ann(R)\big\}\,.$$
It follows that $H_-=L_-$ as $L_-\subset T\mathcal{R}^{\perp}\otimes\C$; see Proposition \ref{prop c}, and so, $H_-$ is Courant involutive. Observe that, by Remark \ref{imp rmk1}, $H_-/(T\mathcal{R}\otimes\C)=E^{(1,0)}$, and so
$$K_-=\pr_{TM}(E^{(1,0)})\,,\,\,\,\text{and}\,\,\,\Delta_-=\pr_{TM}(E^{(1,0)})\cap \pr_{TM}(\overline{E^{(1,0)}})\,,$$ where $\pr_{TM}\,,K_-$, and $\Delta_-$ are as in \eqref{pr}, \eqref{k} and \eqref{k1}, respectively. Therefore, by Theorem \ref{thm c1}, $M/\mathcal{R}$ admits a generalized comple (GC) structure whenever it is a smooth manifold. In addition, if $R$ is complete and satisfies $\mathcal{L}_R\eta=0$, then $M\rightarrow M/\mathcal{R}$ becomes a principal bundle by Theorem \ref{thmc2}. Particularly, $\ker(\eta)=\Psi^*T(M/\mathcal{R})$. Hence, the following corollary is obtained, after applying Theorem \ref{thmc2}, as an extension of \cite[Corollary 5.9]{wade12}.
\begin{cor}\label{cor2}
Given any regular generalized contact manifold $(M, \Phi, E_{\pm})$ of Poon-Wade type with the associated foliaiton $\mathcal{R}$ (cf. Definition \ref{def:reg}), the following holds.
\begin{enumerate}
\setlength\itemsep{0.3em}
    \item $M/\mathcal{R}$ is a GC manifold.
    \item $M\rightarrow M/\mathcal{R}$ is a principal $\R$-or $S^1$-bundle, provided $R$ is complete and satisfies $\mathcal{L}_R\eta=0$. Also, $\ker(\eta)=\Psi^*T(M/\mathcal{R})$.
\end{enumerate}
\end{cor}
In the following, we apply Theorem \ref{thmc2} to certain types of generalized contact manifolds of Poon-Wade type.
\begin{example} (Contact case)\label{contact1}
    Let $(M, \Phi, E_{\pm})$ be a regular generalized contact manifold corresponding to a regular cooriented contact manifold $(M^{2n+1},\xi,\alpha)$ as in Example \ref{contact eg}. It follows that
    $$E^{(1,0)}=\big\{Y+\beta-i(\pi(\beta)+d\alpha(Y))\,|\,Y\in\ker(\alpha),\beta\in\ann(\xi)\big\}\,.$$ Setting $\beta=0$, we get that $\pr_{TM}(E^{(1,0)})\subseteq\ker(\alpha)\otimes\C$ where $\pr_{TM}$ is the projection (cf. \eqref{pr}). Other direction follows from the fact that $Y-i\,d\alpha(Y)\in E^{(1,0)}$ for all $Y\in\ker(\alpha)$. Therefore $$\pr_{TM}(E^{(1,0)})=\ker(\alpha)\otimes\C\,\,\,\text{and so,}\,\,\,\Delta_-=\ker(\alpha)\,.$$
    The contact condition $\alpha\wedge(d\alpha)^n\neq 0$ implies that $d\alpha|_{\ker(\alpha)}$ is nondegenerate. Since the forms $\alpha$, $d\alpha$ are invariant with respect to the Reeb vector field $\xi$, the endomorphism $\Phi$ is also invariant. More precisely, $[\xi,\alpha]=0$, $\mathcal{L}_\xi d\alpha=0$, and $\mathcal{L}_\xi\Phi=0$. Thus, by Corollary \ref{cor2}, $M\xrightarrow{\Psi} M/\mathcal{R}$ is a principal $\R$-or $S^1$-bundle over the GC manifold $(M/\mathcal{R},\mathcal{J}_{red})$ with $\alpha$ as its connection form, provided $\xi$ is complete. Consequently, by Theorem \ref{thmc2}, the curvature $\omega$ is a symplectic from on $M/\mathcal{R}$, since $\Psi^*(T(M/\mathcal{R}))=\ker(\alpha)$. Furthermore, since $\rho_-\equiv n$ on $M$, $\mathcal{J}_{red}$ is regular GC structure of type $0$, and has the form
$$\mathcal{J}_{red}=\begin{pmatrix}
    0    &-\omega^{-1}\\
    \omega    &0
\end{pmatrix}\,;$$ see Example \ref{symplectic eg}. This essentially recovers \cite[Theorem~1.1]{grabo25} from a more general viewpoint.
\end{example}
\begin{example}(Normal almost contact case)\label{nrml1}
Suppose $(M, \Phi, E_{\pm})$ represent a generalized contact manifold associated to a normal almost contact manifold $(M^{2n+1},\varphi,\xi,\alpha)$ as given in Example \ref{almost eg}. Then 
$$E^{(1,0)}=\big\{(Y-i\varphi(Y))+(\beta+i\varphi^*(\beta))\,|\,Y\in\ker(\alpha),\beta\in\ann(\xi)\big\}\,.$$
By Remark \ref{imp rmk2}, it follows that $\varphi(\ker(\alpha))=\ker(\alpha)$, since $\varphi^2|_{\ker(\alpha}=-I_{TM}$. So, by setting $\beta=0$, we have $\pr_{TM}(E^{(1,0)})\subseteq\ker(\alpha)\otimes\C$. Similarly, $\pr_{TM}(E^{(0,1)})\subseteq\ker(\alpha)\otimes\C$. In fact, we can express $\pr_{TM}(E^{(1,0)})$ in the form
$$\pr_{TM}(E^{(1,0)})=\big\{(Y-i\varphi(Y)\,|\,Y\in\ker(\alpha)\big\}\,,$$ which implies that $\pr_{TM}(E^{(1,0)})$ is a complex subbundle of rank $n$, and hence $$\pr_{TM}(E^{(1,0)})\oplus\pr_{TM}(E^{(0,1)})=\ker(\alpha)\otimes\C\,.$$ Since $(M, \Phi, E_{\pm})$ also normal, by Proposition \ref{strng}, we have $[\xi,\alpha]=0$ and $$[C^\infty(E^{(1,0)}, C^\infty(E^{(1,0)})]\subseteq C^\infty(E^{(1,0)})\,.$$ Therefore, $\varphi$ is a complex structure on $\ker(\alpha)$ with the $+i$-eigenbundle $\pr_{TM}(E^{(1,0)})$. Furthermore, by Theorem \ref{thmc2}, $M\xrightarrow{\Psi} M/\mathcal{R}$ is a principal $\R$-or $S^1$-bundle and $M/\mathcal{R}$ admits a GC structure $\mathcal{J}_{red}$, provided $\xi$ is complete and its induced foliation $\mathcal{R}$ is regular. Also, $\mathcal{J}_{red}$ is a regular GC structure of type $n$, and its $+i$-eigenbundle $L_{red}$ is of the form
$$L_{red}=T^{1,0}(M/\mathcal{R})\oplus(T^{0,1}(M/\mathcal{R}))^{*}\,,$$ where $T^{1,0}(M/\mathcal{R})$ denotes the $+i$-eigenbundle of the complex structure on $M/\mathcal{R}$; see Example \ref{complx eg}. This complex structure is induced from $\ker(\alpha)$, using the identification $\Psi^*(T(M/\mathcal{R}))=\ker(\alpha)$, since $\alpha$ is a connection $1$-form. Now, $d\alpha(X,Y)=0$ for all $X,Y\in C^\infty(\pr_{TM}(E^{(1,0)}))$. Since $d\alpha$ is a real form, it follows that 
$$d\alpha\in C^\infty\big(\pr_{TM}(E^{(1,0)})^*\otimes\pr_{TM}(E^{(0,1)})^*\big)\,.$$ Therefore, the curvature form $\omega$ is of type $(1,1)$. This provides a new proof of a well-known result by Morimoto \cite[Theorem~1]{mori}.
\end{example}
\begin{example}(Almost cosymplectic case)\label{cosymp1}
    Let $(M, \Phi, E_{\pm})$ be a regular generalized contact manifold associated with an almost cosymplectic manifold $(M^{2n+1},\alpha,\theta)$, as defined in Example \ref{cosym eg}. Then, by Corollary \ref{cor2}, $M\xrightarrow{\Psi}M/\mathcal{R}$ is a principal bundle over the GC manifold $(M, \mathcal{J}_{red})$ with $\alpha$ as its connection, provided $\xi$ is complete and $[\xi,\alpha]=0$. Here $\xi$ is the unique nowhere vanishing vector field associated with $(\alpha,\theta)$. So,
    $$E^{(1,0)}=\big\{(Y-i\pi(\beta))+(\beta-i\theta(Y))\,|\,Y\in\ker(\alpha),\beta\in\ann(\xi)\big\}\,.$$
    It follows that $\ker(\alpha)\otimes\C\subseteq\pr_{TM}(E^{(1,0)})$, and hence $\pr_{TM}(E^{(1,0)})=\ker(\alpha)\otimes\C$. The almost cosymplectic condition $\alpha\wedge\theta^n\neq 0$ implies that the restriction $\theta|_{\ker(\alpha)}$ is nondegenerate. Since $d\theta=0=\mathcal{L}_\xi\theta$, there exists a symplectic form $\tilde{\omega}$ such that $\theta=\Psi^*\tilde{\omega}$. Consequently, $\mathcal{J}_{red}$ is regular GC structure of type $0$, since $\rho_-\equiv n$ on $M$, and and it is given by the form
$$\mathcal{J}_{red}=\begin{pmatrix}
    0    &-\tilde{\omega}^{-1}\\
    \tilde{\omega}    &0
\end{pmatrix}\,;$$ see Example \ref{symplectic eg}. In this case, condition \eqref{conditn} does not fully determine the curvature $\omega$. It may vanish, as in cosymplectic manifolds, or be nonzero, as when $\alpha$ is a contact form.
\end{example}
\begin{remark}
    Theorem \ref{thmc2} does not imply that the cohomology class of the curvature form, when restricted to the symplectic leaves of $\mathcal{J}_{red}$, coincides with the cohomology class of the symplectic forms on those leaves. This is clear from Examples \ref{nrml1}-\ref{cosymp1}. In particular, Example \ref{nrml1} demonstrates that even when the symplectic leaves are zero-dimensional, the curvature class may still be nonzero. On the other hand, Example \ref{cosymp1} shows that for cosymplectic manifolds the curvature always vanishes, whereas the symplectic forms on the leaves are nonzero. It is an interesting problem to fully determine the most general yet natural conditions on the total space under which the curvature induces the symplectic foliation.
\end{remark}
\subsection{Construction for new examples of non Poon-Wade types}\label{non pw}
Let $M$ be a smooth manifold of dimension $2n+1$. Let $E_{\pm}\in C^\infty(TM\oplus T^*M)$ be smooth sections that satisfies 
$$\pr_{TM}(E_\pm)\neq0\neq\pr_{T^*M}(E_\pm)\,,\langle E_\pm,E_\pm\rangle=0\,,\,\,\,\text{and}\,\,\,2\langle E_+,E_-\rangle=1\,.$$
Here $\pr_{TM}$, $\pr_{T^*M}$ are projection onto $TM$ and $T^*M$, respectively; see \eqref{pr}. Let $V:=\spn\{E_+,E_-\}$ and $V^\perp\subset C^\infty(TM\oplus T^*M)$ be the orthogonal complement (with respect to the bilinear form \eqref{bilinear}); see Definition \ref{imp def}. Since $V$ is of rank $2$, $V^\perp$ is a subbundle of rank $4n$. Suppose $V^\perp$ admits a generalized complex (GC) structure $\mathcal{J}$, that is, 
$$\mathcal{J}:V^\perp\rightarrow V^\perp\,,$$ and satisfies all of the properties in Definition \ref{gcs}. Define
\begin{equation}\label{neweg1}
\Phi=
  \begin{cases}
    \mathcal{J}& \text{on}\quad V^\perp\,;\\
    0 & \text{on}\quad V\,.
\end{cases}  
\end{equation}
Then $(M,\Phi,E_\pm)$ is a generalized almost contact manifold which is not of Poon-Wade type. It follows that the corresponding bundles $L_\pm$ (see \eqref{contact dirac}) can be written in the form
$$L_\pm=(L_{E_\pm}\otimes\C)\oplus\ker(\mathcal{J}- i)\,.$$ Since $\mathcal{J}$ is a GC structure, its eigenbundles $\ker(\mathcal{J}\pm i)$ are Courant involutive. Therefore, $(\Phi,E_\pm)$ is a generalized contact structure if and only if  any of the following conditions hold,
\begin{equation}\label{con1}
\begin{aligned}
    &\bullet\quad [E_+,\ker(\mathcal{J}- i)]\subseteq C^\infty(L_+)\,;\\
    &\bullet\quad [E_-,\ker(\mathcal{J}- i)]\subseteq C^\infty(L_-)\,.
\end{aligned}
\end{equation}
We now apply this construction in the subsequent example.
\begin{example}(For $\R^{2n+1}$ and $\T^{2n+1}$)\label{neweg2}
  Let $M^{2n+1}$ be a smooth parallelizable manifold. Let $\{X_k\}^{2n+1}_{k=1}$ be a global frame of $TM$, and $\{\alpha_l\}^{2n+1}_{l=1}$ denotes the corresponding dual frame. Without loss of generality, set 
$$E_+:=X_1+\alpha_2\,\,\,\text{and}\,\,\,E_-:=\frac{1}{2}(X_2+\alpha_1)\,.$$
Clearly, $\langle E_\pm,E_\pm\rangle=0$ and $2\langle E_+,E_-\rangle=1$. Then, $V=\spn\{E_+,E_-\}$ and $$V^\perp=\spn\{X_k\,,\alpha_l\}_{k,l\neq 1,2}\oplus\spn\{X_2-\alpha_1\,,X_1-\alpha_2\}\,.$$
Since $\dim_\R V^\perp=4n$, it always admits a regular GC structure $\mathcal{J}$ by \cite[Proposition~4.5]{Gua}; see also Theorem \ref{darbu thm}. Therefore, $(\Phi,E_\pm)$ is a generalized almost contact structure where $\Phi$ as defined in \eqref{neweg1}. This example provides a direct generalization of Example \ref{new eq}.
\medskip

Let us construct a specific example of such a $\mathcal{J}$ as follows. Choose a complex structure $\phi$ on $\spn\{X_k\}_{k=4}^{2n+1}$ such that $$\ker(\phi-i)=\spn\left\{X_k-iX_{k+2}\right\}^{2n-1}_{k=4}\,,$$ and define $\mathcal{J}_1=\begin{pmatrix}
 \phi     &0 \\
    
    0        &-\phi^{*}
\end{pmatrix}$ on $\spn\{X_k,\alpha_l\}_{k,l=4}^{2n+1}$. Then as in Example \ref{new eq}, consider a map $\mathcal{J}_2$ on $\spn\{X_1-\alpha_2, X_2-\alpha_1,X_3,\alpha_3\}$, defined by  
\begin{align*}
&\mathcal{J}_2(X_1-\alpha_2)=\sqrt{2}\alpha_3\,,\,\mathcal{J}_2(X_3)=\frac{1}{\sqrt{2}}(X_2-\alpha_1)\,;\\
&\mathcal{J}_2(X_2-\alpha_1)=-\sqrt{2} X_3\,,\,\,\text{and}\,\,\mathcal{J}_2(\alpha_3)=-\frac{1}{\sqrt{2}}(X_1-\alpha_2)\,.
\end{align*}
Note that 
\begin{equation}\label{ker1}
    \begin{aligned}
       &\ker(\mathcal{J}_2- i)=\spn\left\{X_1-\alpha_2-\sqrt{2}i\alpha_3\,,X_2-\alpha_1+\sqrt{2}iX_3\right\}\,,\\
       &\ker(\mathcal{J}_2+i)=\spn\left\{X_1-\alpha_2+\sqrt{2}i\alpha_3\,,X_2-\alpha_1-\sqrt{2}iX_3\right\}\,.
    \end{aligned}
\end{equation}
Therefore, define
\begin{equation}\label{neweq1}
\mathcal{J}=
  \begin{cases}
    \mathcal{J}_1& \text{on}\quad \spn\{X_k\,,\alpha_l\}_{k,l=4}^{2n+1};\\
    \mathcal{J}_2 & \text{on}\quad \spn\{X_1-\alpha_2\,,X_2-\alpha_1,X_3\,,\alpha_3\}\,.
\end{cases}  
\end{equation}
Hence $(\Phi,E_\pm)$ is a generalized almost contact structure with respect to $\mathcal{J}$ in \eqref{neweq1}. Also
\begin{align*}
    \ker(\mathcal{J}-i)&=\ker(\mathcal{J}_1-i)\oplus\ker(\mathcal{J}_2-i)\\
    &=\spn\left\{X_k-iX_{k+2}\,,\alpha_k-i\alpha_{k+2}\right\}^{2n-1}_{k=4}\\
    &\quad\oplus\spn\left\{X_1-\alpha_2-\sqrt{2}i\alpha_3\,,X_2-\alpha_1+\sqrt{2}iX_3\right\}
\end{align*}
Thus, we get
\begin{align*}
    &[X_1+\alpha_2,X_k-iX_{k+2}]=[X_1,X_k]-i[X_1,X_{k+2}]-\mathcal{L}_{X_k}\alpha_2+i\mathcal{L}_{X_{k+2}}\alpha_2\,;\\
    &[X_1+\alpha_2,\alpha_k-i\alpha_{k+2}]=\mathcal{L}_{X_1}\alpha_k-i\mathcal{L}_{X_1}\alpha_{k+2}\,;\\
    &[X_1+\alpha_2,X_1-\alpha_2-\sqrt{2}i\alpha_3]=-2\mathcal{L}_{X_1}\alpha_2-\sqrt{2}i\mathcal{L}_{X_1}\alpha_3;\\
    &[X_1+\alpha_2,X_2-\alpha_1+\sqrt{2}iX_3]=[X_1,X_2]+\sqrt{2}i[X_1,X_3]-\mathcal{L}_{X_1}\alpha_1-\mathcal{L}_{X_2}\alpha_2-\sqrt{2}i\mathcal{L}_{X_3}\alpha_3\,.
\end{align*}
Likewise, we can compute the Courant brackets for $E_-$. Therefore, whether $(\Phi,E_\pm)$ defines a generalized contact structure depends only on the choice of a global frame and co-frame. In particular, suppose $[X_k,X_l]=0$, $[X_k,\alpha_l]=0$ for all $k,l\in\{1,\ldots,2n+1\}$. It follows that $(\Phi,E_\pm)$ is a normal generalized contact structure. For example, if $M=\R^{2n+1}$, $X_k=\frac{\partial}{\partial x_k}$ and $\alpha_l=dx_l$, then $(\Phi,E_\pm)$ decends to a normal generalized contact structure on $\T^{2n+1}=\R^{2n+1}/\Z^{2n+1}$. Moreover, the generalized almost contact structure in Example \ref{new eq}, is also normal generalized contact structure.
\end{example}

\begin{example}\label{neweg2.2}
Let $M=\R^2\times\R$ with coordinates $(x,y,t)$. Consider the symplectic form $\omega=dx\wedge dy$ on $\R^2$, and let $\mathcal{J}$ denote the associated generalized complex (GC) structure defined by $\omega$, as in Example \ref{symplectic eg}. Particularly, we have
$$\mathcal{J}\left(\frac{\partial}{\partial x}\right)=-dy\,,\mathcal{J}\left(\frac{\partial}{\partial y}\right)=dx\,,\mathcal{J}(dx)=\left(\frac{\partial}{\partial y}\right)\,,\,\,\,\text{and}\,\,\,\mathcal{J}(dy)=-\left(\frac{\partial}{\partial x}\right)\,.$$
Extend on $M$ by $$\mathcal{J}(dt)=\mathcal{J}\left(\frac{\partial}{\partial t}\right)=0\,.$$
Denote this endomorphism by $\Phi$. Set $$E_+=\frac{\partial}{\partial t}\quad\text{and}\quad E_-=\cos(2\pi t)\frac{\partial}{\partial x}+dt\,.$$
It follows that $(\Phi,E_\pm)$ defines a generalized almost contact structure on $M$. Let $V=\spn\{E_+,E_-\}$. Then the orthogonal complement $V^\perp$ (cf. Definition \ref{imp def}) can be written as
$$V^\perp=\spn\left\{\frac{\partial}{\partial x}\,,\frac{\partial}{\partial y}\,,dx-\cos(2\pi t)\frac{\partial}{\partial t}\,,dy\right\}\,.$$
\end{example}
Therefore, the associated $+i$-eigenbundle $\Phi$ is given by
$$\ker(\Phi-i)=\spn\left\{\frac{\partial}{\partial x}+idy\,,\frac{\partial}{\partial y}-idx\,,dx-\cos(2\pi t)\frac{\partial}{\partial t}-i\frac{\partial}{\partial y}\,,dy+i\frac{\partial}{\partial x}\right\}\,.$$
It is then immediate that $L_+$ (cf. \eqref{contact dirac}) is Courant involutive, and consequently, $(\Phi,E_\pm)$ is a generalized contact structure. Moreover, 
\begin{align*}
 [E_-\,,dx-\cos(2\pi t)\frac{\partial}{\partial t}-i\frac{\partial}{\partial y}]&=[\cos(2\pi t)\frac{\partial}{\partial x}\,,\cos(2\pi t)\frac{\partial}{\partial t}]+[\cos(2\pi t)\frac{\partial}{\partial t}\,,dt]\\
 &=\pi\left(\sin(4\pi t)\frac{\partial}{\partial x}+dt\right)\,.
\end{align*}
If possible, let $\sin(4\pi t)\frac{\partial}{\partial x}+dt\in C^\infty(L_-)$. Then, for some smooth functions $f,g$, we have $\sin(4\pi t)\frac{\partial}{\partial x}+dt=fE_-+gV$ where $V\in\ker(\Phi-i)$. Then,
$$\frac{f}{2}=\left\langle\sin(4\pi t)\frac{\partial}{\partial x}+dt,\frac{\partial}{\partial t}\right\rangle=1\,,$$
implying,
$$(\sin(4\pi t)-2\cos(2\pi t))\frac{\partial}{\partial x}-dt\in\ker(\Phi-i)\,,$$ which is not possible since $\left\langle(\sin(4\pi t)-2\cos(2\pi t))\frac{\partial}{\partial x}-dt\,,E_+\right\rangle\neq\,0$. Also $[E_+,E_-]\neq 0$. Hence $L_-$ is not Courant involutive, implying $(\Phi,E_\pm)$ is not strong, and normal generalized contact structure; see Definition \ref{def:contct}.
\medskip

Furthermore, since $\Phi$ and $E_\pm$ are translation invariant, the pair $\Phi, E_\pm$  descends to generalized contact structures on $\T^2\times\R$ $(\cong\R^2/\Z^2\times\R)$ , and $\T^2\times S^1$ $(\cong\R^2/\Z^2\times\R/\Z)$. These structures are not strong, normal and also, are not of Poon-Wade type.

\medskip
{\bf Acknowledgement.} The author would like to thank the anonymous reviewers for the careful reading of the paper and for the constructive comments and valuable suggestions.


\begin{thebibliography}{}
\setlength\itemsep{0.6em}

\bibitem{aldi}Aldi, M. \& Grandini, D. Generalized contact geometry and T-duality. {\em J. Geom. Phys.}. \textbf{92} pp. 78-93 (2015), \url{https://doi.org/10.1016/j.geomphys.2015.02.007}.

\bibitem{bailey}Bailey, M. Symplectic foliations and generalized complex structures. {\em Canad. J. Math.}. \textbf{66}, 31-56 (2014), \url{https://doi.org/10.4153/CJM-2013-007-6}.

\bibitem{blair}Blair, D. Riemannian geometry of contact and symplectic manifolds. (Birkhäuser Boston, Ltd., Boston, MA,2010), \url{https://doi.org/10.1007/978-0-8176-4959-3}.

\bibitem{bw}Boothby, W. \& Wang, H. On contact manifolds. {\em Ann. Of Math. (2)}. \textbf{68} pp. 721-734 (1958), \url{https://doi.org/10.2307/1970165}.

\bibitem{bott71}Bott, R. Lectures on characteristic classes and foliations. {\em Lectures On Algebraic And Differential Topology (Second Latin American School In Math., Mexico City, 1971)}. \textbf{Vol. 279} pp. 1-94 (1972), Notes by Lawrence Conlon, with two appendices by J. Stasheff.

\bibitem{burst}Bursztyn, H., Cavalcanti, G. \& Gualtieri, M. Reduction of Courant algebroids and generalized complex structures. {\em Adv. Math.}. \textbf{211}, 726-765 (2007), \url{https://doi.org/10.1016/j.aim.2006.09.008}.

\bibitem{zambon1}Calvo, I., Falceto, F. \& Zambon, M. Deformation of Dirac structures along isotropic subbundles. {\em Rep. Math. Phys.}. \textbf{65}, 259-269 (2010), \url{https://doi.org/10.1016/S0034-4877(10)80020-5}.

\bibitem{gei08}Geiges, H. An introduction to contact topology. (Cambridge University Press, Cambridge, 2008), \url{https://doi.org/10.1017/CBO9780511611438}.

\bibitem{gomez}Gomez, R. \& Talvacchia, J. On products of generalized geometries. {\em Geom. Dedicata}. \textbf{175} pp. 211-218 (2015), \url{https://doi.org/10.1007/s10711-014-0036-6}.

\bibitem{grabo25}Grabowska, K. \& Grabowski, J. The regularity and products in contact geometry. {\em Annali Di Matematica Pura Ed Applicata (1923 -)}. (2025, 11), \url{https://doi.org/10.1007/s10231-025-01631-7}.

\bibitem{Gua} Gualtieri, M. Generalized complex geometry, DPhil thesis, University of Oxford, 2003, \url{arXiv:math.DG/0401221}.

\bibitem{Gua2} Gualtieri, M. Generalized complex geometry. Ann. of Math. (2) 174 (2011), no. 1, 75--123.

\bibitem{hit1}Hitchin, N. Generalized Calabi-Yau manifolds. {\em Q. J. Math.}. \textbf{54}, 281-308 (2003), \url{https://doi.org/10.1093/qjmath/54.3.281}.

\bibitem{pw2}Iglesias-Ponte, D. \& Wade, A. Contact manifolds and generalized complex structures. {\em J. Geom. Phys.}. \textbf{53}, 249-258 (2005), \url{https://doi.org/10.1016/j.geomphys.2004.06.006}.

\bibitem{koba}Kobayashi, S. Principal fibre bundles with the 1-dimensional toroidal group. {\em Tohoku Math. J. (2)}. \textbf{8} pp. 29-45 (1956), \url{https://doi.org/10.2748/tmj/1178245006}.

\bibitem{kobayashi14}Kobayashi, S. Differential geometry of complex vector bundles. (Princeton University Press, Princeton, NJ), Reprint of the 1987 edition.

\bibitem{lang}Lang, S. Differential and Riemannian manifolds. (Springer-Verlag, New York,1995), \url{https://doi.org/10.1007/978-1-4612-4182-9}.

\bibitem{lee}Lee, J. Introduction to smooth manifolds. (Springer, New York,2013).

\bibitem{meig}Meigniez, G. Submersions, fibrations and bundles. {\em Trans. Amer. Math. Soc.}. \textbf{354}, 3771-3787 (2002), \url{https://doi.org/10.1090/S0002-9947-02-02972-0}.

\bibitem{mori1}Morimoto, A. On normal almost contact structures. {\em J. Math. Soc. Japan}. \textbf{15} pp. 420-436 (1963), \url{https://doi.org/10.2969/jmsj/01540420}.

\bibitem{mori}Morimoto, A. On normal almost contact structures with a regularity. {\em Tohoku Math. J. (2)}. \textbf{16} pp. 90-104 (1964), \url{https://doi.org/10.2748/tmj/1178243735}.

\bibitem{pw3}Poon, Y. \& Wade, A. Approaches to generalize contact structures. {\em Pure Appl. Math. Q.}. \textbf{6}, 603-622 (2010), \url{https://doi.org/10.4310/PAMQ.2010.v6.n2.a12}.

\bibitem{pw}Poon, Y. \& Wade, A. Generalized contact structures. {\em J. Lond. Math. Soc. (2)}. \textbf{83}, 333-352 (2011), \url{https://doi.org/10.1112/jlms/jdq069}.

\bibitem{sekiya}Sekiya, K. Generalized almost contact structures and generalized Sasakian structures, {\em Osaka J. Math.} {\bf 52} (2015), no.~1, 43--59; MR3326601

\bibitem{thurston}Thurston, W. A generalization of the Reeb stability theorem. {\em Topology}. \textbf{13} pp. 347-352 (1974), \url{https://doi.org/10.1016/0040-9383(74)90025-1}.

\bibitem{tu17}Tu, L. Differential geometry: Connections, curvature, and characteristic classes (Springer, Cham, 2017), \url{https://doi.org/10.1007/978-3-319-55084-8}.

\bibitem{vais1}Vaisman, I. Dirac structures and generalized complex structures on $TM\times\R^h$. {\em Adv. Geom.}. \textbf{7}, 453-474 (2007), \url{https://doi.org/10.1515/ADVGEOM.2007.029}.

\bibitem{vais2}Vaisman, I. From generalized K\"{a}hler to generalized Sasakian structures. {\em J. Geom. Symmetry Phys.}. \textbf{18} pp. 63-86 (2010).

\bibitem{vita}Vitagliano, L. \& Wade, A. Generalized contact bundles. {\em C. R. Math. Acad. Sci. Paris}. \textbf{354}, 313-317 (2016), \url{https://doi.org/10.1016/j.crma.2015.12.009}.

\bibitem{wade12}Wade, A. Local structure of generalized contact manifolds. {\em Differential Geom. Appl.}. \textbf{30}, 124-135 (2012), \url{https://doi.org/10.1016/j.difgeo.2011.11.009}.

\bibitem{wright}Wright, K. Generalised contact geometry as reduced generalised complex geometry. {\em J. Geom. Phys.}. \textbf{130} pp. 331-348 (2018), \url{https://doi.org/10.1016/j.geomphys.2018.04.009}.

\bibitem{zambon}Zambon, M. Reduction of branes in generalized complex geometry. {\em J. Symplectic Geom.}. \textbf{6}, 353-378 (2008), \url{https://doi.org/10.4310/jsg.2008.v6.n4.a1}.
\end{thebibliography}
\end{document}